\pdfoutput=1
\documentclass[a4paper]{article}
\usepackage{amsmath,amssymb,amsthm}
\usepackage[english]{babel}
\usepackage[T1]{fontenc}
\usepackage{lmodern}
\usepackage[utf8]{inputenc}
\usepackage[babel=true]{microtype}
\usepackage[autostyle=true]{csquotes}

\usepackage[backref]{hyperref}

\usepackage{hyperref}
\usepackage{mathtools}
\usepackage{amscd}
\usepackage{tikz-cd}
\usepackage{thmtools}
\usepackage{amsrefs}

\usepackage[shortlabels]{enumitem}
\newlist{myenumi}{enumerate}{1}
\setlist[myenumi,1]{label=\upshape(\roman*)}
\newlist{myenuma}{enumerate}{1}
\setlist[myenuma,1]{label=\upshape(\alph*)}
\newlist{myenumext}{enumerate}{1}
\setlist[myenumext,1]{label=\upshape(\Roman*)}
\setlist[enumerate]{label=\upshape(\arabic*)}

\newcounter{commentcounter}

\usepackage{ifthen,srcltx,color}
\newcommand{\showcomments}{yes}

\newsavebox{\commentbox}
\newenvironment{com}%
{\ifthenelse{\equal{\showcomments}{yes}}%
{\raisebox{0.7mm}{\small\cref{com:\arabic{commentcounter}}}
        \begin{lrbox}{\commentbox}
        \begin{minipage}[t]{1.2in}\raggedright\sffamily\tiny
        \raisebox{0.7mm}{\thecommentcounter}\label{com:\arabic{commentcounter}}}
{\begin{lrbox}{\commentbox}}}
{\ifthenelse{\equal{\showcomments}{yes}}
{\end{minipage}\end{lrbox}\marginpar{\usebox{\commentbox}}}
{\end{lrbox}}}

\numberwithin{equation}{section}

\declaretheoremstyle[headformat=swapnumber,headfont=\bfseries,bodyfont=\itshape]{theorem}
\declaretheorem[name=Theorem,numberlike=equation,style=theorem]{theorem}
\declaretheorem[name=Lemma,numberlike=theorem,style=theorem]{lemma}
\declaretheorem[name=Corollary,numberlike=theorem,style=theorem]{corollary}
\declaretheorem[name=Proposition,numberlike=theorem,style=theorem]{proposition}

\declaretheoremstyle[headformat=swapnumber, bodyfont=\normalfont]{definition}
\declaretheorem[name=Definition,numberlike=theorem,style=definition]{definition}
\declaretheorem[name=Notation,numberlike=theorem,style=definition]{notation}

\declaretheorem[name=Setup,numberlike=theorem,style=definition]{setup}

\declaretheoremstyle[headformat=swapnumber,headfont=\itshape, bodyfont=\normalfont]{remark}
\declaretheorem[name=Example,numberlike=theorem,style=remark]{example}

\declaretheorem[name=Remark,numberlike=theorem,style=remark]{remark}

\usepackage[capitalise]{cleveref}
\crefname{setup}{Setup}{Setups}
\crefname{lemma}{Lemma}{Lemmas}
\crefname{diagram}{Diagram}{Diagram}
\crefname{commentcounter}{}{}
\crefdefaultlabelformat{#2\textup{#1}#3}

\usepackage{xparse}

\newcommand{\reals}{\mathbb{R}}

\newcommand{\integers}{\mathbb{Z}}
\newcommand{\rationals}{\mathbb{Q}}

\DeclareMathOperator{\id}{id}

\newcommand{\into}{\hookrightarrow}
\newcommand{\onto}{\twoheadrightarrow}
\newcommand{\iso}{\cong}

\DeclareMathOperator{\ch}{ch}  
\DeclareMathOperator{\ph}{ph}  

\DeclareMathOperator{\im}{im}      

\DeclareMathOperator{\spin}{spin}
\DeclareMathOperator{\spinC}{\spin^{\mathrm{c}}}

\DeclareMathOperator{\colim}{colim}
\newcommand{\sphere}{\mathrm{S}}

\NewDocumentCommand \KK {} {\mathrm{KKO}}

\NewDocumentCommand \RKO {} {\mathrm{KO}}
\NewDocumentCommand \K {} {\mathrm{K}}
\NewDocumentCommand \HZ {} {\mathrm{H}}
\NewDocumentCommand \KO {} {\mathrm{KO}}
\NewDocumentCommand \SpinBordism {} {\Omega^{\mathrm{\spin}}}

\newcommand{\ccomp}[2]{#1 \setminus #2} 

\newcommand*\ucov[1]{{\mathpalette\wthelper{#1}}} 

\makeatletter
\newcommand*\wthelper[2]{%
        \hbox{\dimen@\accentfontxheight#1%
                \accentfontxheight#11.25\dimen@
                $\m@th#1\widetilde{#2}$%
                \accentfontxheight#1\dimen@
        }%
}
\newcommand*\accentfontxheight[1]{%
        \fontdimen5\ifx#1\displaystyle
                \textfont
        \else\ifx#1\textstyle
                \textfont
        \else\ifx#1\scriptstyle
                \scriptfont
        \else
                \scriptscriptfont
        \fi\fi\fi3
}

\newcommand*\printstyle[1]{%
        \ifx#1\displaystyle
                textfont
        \else\ifx#1\textstyle
                textfont
        \else\ifx#1\scriptstyle
                scriptfont
        \else
                scriptscriptfont
        \fi\fi\fi
}
\makeatother

\newcommand{\Eub}{\underline{\mathrm{E}}}
\DeclareMathOperator{\Ind}{Ind}
\DeclareMathOperator{\ii}{i}

\NewDocumentCommand{\pt}{}{\mathrm{pt}}

\newcommand{\Bfree}{\mathrm{B}}
\newcommand{\Efree}{\mathrm{E}}

\DeclareMathOperator{\tn}{\mathsf{t}}

\NewDocumentCommand \numbers {m} {\mathbb{#1}}
\NewDocumentCommand \Q {} {\numbers{Q}}
\NewDocumentCommand \R {} {\numbers{R}}
\NewDocumentCommand \C {} {\numbers{C}}
\NewDocumentCommand \N {} {\numbers{N}}
\NewDocumentCommand \Z {} {\numbers{Z}}
\NewDocumentCommand \Cz {} {\mathrm{C}_{0}}

\NewDocumentCommand \Cb {} {\mathrm{C}_{\mathrm{b}}}

\NewDocumentCommand \blank {} {{-}}
\NewDocumentCommand \Cstar {} {\mathrm{C}^\ast}
\NewDocumentCommand \red {} {\mathrm{red}}
\NewDocumentCommand \CstarRed {} {\Cstar_{\red}}
\NewDocumentCommand \CstarMax {} {\Cstar_{\max}}
\NewDocumentCommand \lf {} {\mathrm{lf}}
\NewDocumentCommand \CP {} {\C\mathrm{P}}

\NewDocumentCommand{\parensup}{m}{\textup{(}#1\textup{)}} 

\NewDocumentCommand \disk {s} {
  \IfBooleanTF{#1} {
    \mathring{\mathrm{D}}
  } {
    \mathrm{D}
  }
}

\NewDocumentCommand \myEgenSymb {m m m m m} {
  {#5}_{#3}
  ^{
    \IfBooleanT{#1}{
      #4
      \IfNoValueF{#2}{,}
    }
    \IfNoValueF{#2}{#2}
  }
}

\NewDocumentCommand \mycoEgenSymb {m m m m m} {
  {#5}^{#3}
  _{
    \IfBooleanT{#1}{
      #4
      \IfNoValueF{#2}{,}
    }
    \IfNoValueF{#2}{#2}
  }
}

\NewDocumentCommand \myE {s o m} {
  \myEgenSymb{#1}{#2}{#3}{\mathrm{lf}}{E}
}

\NewDocumentCommand \myEvert {s o m} {
  \myEgenSymb{#1}{#2}{#3}{\mathrm{vlf}}{E}
}

\NewDocumentCommand \mycoE {s o m} {
  \mycoEgenSymb{#1}{#2}{#3}{\mathrm{c}}{E}
}

\NewDocumentCommand \mycoEvert {s o m} {
  \mycoEgenSymb{#1}{#2}{#3}{\mathrm{vc}}{E}
}

\NewDocumentCommand \myZ {s o} {
  \mathrm{Z}
  \IfNoValueF{#2}{^{(#2)}}
}

\newcommand{\forget}[1]{}

\allowdisplaybreaks[2]

\begin{document}


\title{Transfer maps in generalized group homology via submanifolds}

\author{
Martin Nitsche
\thanks{
  e-mail:~\href{mailto:martin.nitsche@tu-dresden.de}{martin.nitsche@tu-dresden.de}
  \newline
  M.~Nitsche was partially supported by the German Research Foundation, DFG,
  under the Research Training Group \enquote{Mathematical Structures in Modern Quantum Physics};
  and partially supported by the European Research Council, ERC,
  under the Consolidator Grant \enquote{Groups, Dynamics, and Approximation}
}\\
\small
Institut f\"ur Geometrie\\[-0.15cm]
\small
TU Dresden, Germany
\and
Thomas Schick
\thanks{
  e-mail:~\href{mailto:thomas.schick@math.uni-goettingen.de}{tschick@math.uni-goettingen.de}, 
  www:~\href{http://www.uni-math.gwdg.de/schick}{http://www.uni-math.gwdg.de/schick}
}\\
\small
Mathematisches Institut\\[-0.15cm]
\small
Universit\"at G{\"o}ttingen, Germany
\and
Rudolf Zeidler
\thanks{
  e-mail:~\href{mailto:math@rzeidler.eu}{math@rzeidler.eu}, 
  www:~\href{https://www.rzeidler.eu}{https://www.rzeidler.eu}
  \newline
  R.~Zeidler was partially supported by the Deutsche Forschungsgemeinschaft (DFG, German Research Foundation) under Germany's Excellence Strategy EXC 2044--390685587, Mathematics Münster: Dynamics -- Geometry -- Structure.
}\\
\small
Mathematisches Institut\\[-0.15cm]
\small
WWU Münster, Germany
}
\maketitle

\begin{abstract}
  Let $N \subset M$ be a submanifold embedding of spin manifolds of some codimension $k \geq 1$.
  A classical result of Gromov and Lawson, refined by Hanke, Pape and Schick,
  states that $M$ does not admit a metric of positive scalar curvature if $k = 2$ and the Dirac operator of $N$ has non-trivial index, provided that suitable conditions are satisfied.
  In the cases $k=1$ and $k=2$, Zeidler and Kubota, respectively, established more systematic results: There exists a transfer $\mathrm{KO}_\ast(\mathrm{C}^{\ast} \pi_1 M)\to \mathrm{KO}_{\ast - k}(\mathrm{C}^\ast \pi_1 N)$ which maps the index class of $M$ to the index class of $N$.
   The main goal of this article is to construct analogous transfer maps $E_\ast(\mathrm{B}\pi_1M) \to E_{\ast-k}(\mathrm{B}\pi_1N)$ for different generalized homology theories $E$ and suitable submanifold embeddings.
   The design criterion is that it is compatible with the transfer $E_\ast(M) \to E_{\ast-k}(N)$ induced by the inclusion $N \subset M$ for a chosen orientation on the normal bundle.
   Under varying restrictions on homotopy groups and the normal bundle, we construct transfers in the following cases in particular:
   In ordinary homology, it works for all codimensions.
   This slightly generalizes a result of Engel and simplifies his proof.
   In complex K-homology, we achieve it for $k \leq 3$.
   For $k \leq 2$, we have a transfer on the equivariant KO-homology of the classifying space for proper actions.
\end{abstract}

\section{Introduction}
  Our starting point is the following result of Gromov and Lawson~\cite{GromovLawsonEnlarg} and Hanke, Pape, and Schick~\cite{HankePapeSchick}.
  All manifolds we consider are smooth.
  \begin{theorem}\label{theo:obstruc_codim_2}
     Let \(M\) be an $m$-dimensional spin manifold and $N\subset M$ a
     submanifold of codimension two with trivial normal bundle. Assume that the inclusion \(N \hookrightarrow M\) induces an injection
     $\pi_1(N)\to \pi_1(M)$ and a surjection $\pi_2(N)\to\pi_2(M)$.

     If the Rosenberg index \(\alpha(N)\in \KO_{m-2}(\Cstar \pi_1(N))\) is non-zero,
     then $M$ does not admit a Riemannian metric of positive scalar curvature.
   \end{theorem}

   The \enquote{standard} obstruction to the existence of a metric with positive scalar
   curvature on $M$ is the Rosenberg index \(\alpha(M)\in
   \KO_m(\Cstar \pi_1(M))\) \emph{of \(M\)}.
   The surprise is that, here, we can use $\alpha(N)$
as a \enquote{submanifold obstruction}.
   However, the Rosenberg index of \(M\) is conjectured to be the \enquote{universal} index obstruction
to positive scalar curvature on \(M\), see~\cite{SchickICM}*{Conjecture 1.5}.
    Thus in the situation of \cref{theo:obstruc_codim_2}, one expects the non-vanishing of \(\alpha(N)\) to imply the non-vanishing of \(\alpha(M)\).
This was recently proved by Kubota for the Rosenberg index in the maximal
group \(\Cstar\)-algebra~\cite{Kubota}*{Section 6.3}. A somewhat simpler proof
of this theorem is given in \cite{KubotaSchick}, where it is restated as \cite{KubotaSchick}*{Theorem 1.2}.

Moreover, Zeidler \cite{ZeidlerSubmf}*{Theorem 1.7} showed the analog assertion for submanifolds of codimension one, for the reduced $\Cstar$\nobreakdash-algebra and without a condition on $\pi_2$.
Engel~\cite{Engel} and Zeidler~\cite{ZeidlerSubmf} also prove variants of this for other codimensions under strong additional hypotheses.

\medskip

It is our goal to conceptually illuminate the submanifold obstruction results à la \cref{theo:obstruc_codim_2} by constructing \emph{transfer maps}:
  We think of the Rosenberg index of a spin manifold \(M\) as the image of the spin
  bordism fundamental class under the sequence of transformations
  \begin{equation}\label{eq:KO_sequence}
    \SpinBordism_n(M) \to \KO_n(M) \xrightarrow{(c_\Gamma)_*} \KO_n(\Bfree \Gamma) \xrightarrow{\iota_*}
    \KO_n^\Gamma(\Eub\Gamma) \xrightarrow{\mu} \KO_n(\Cstar \Gamma),
  \end{equation}
  where \(\Gamma = \pi_1(M)\),
$c_\Gamma$ is classifying the principal $\Gamma$-bundle $\ucov{M}\to M$,
$\iota$ is the classifying map for proper actions and $\mu$ is the Baum-Connes
assembly map. On purpose, we leave it open which $\Cstar$-completion $\Cstar\Gamma$
of $\reals\Gamma$ we use: if we use $\CstarMax\Gamma$ we get a priori a
stronger invariant, if we use $\CstarRed\Gamma$ the Baum-Connes conjecture
predicts that $\mu$ is an isomorphism.

When $N\subset M$ is a codimension $k$ submanifold (satisfying certain
conditions) and $\pi={i_N}_*(\pi_1(N))<\pi_1(M)$, we ---ideally--- want to construct transfer maps fitting into a commutative diagram
  \begin{equation}\label[diagram]{eqn:KO-extension-diagram}
  \resizebox{\linewidth}{!}{
  \begin{tikzcd}[ampersand replacement=\&]
    \Omega_*^{\spin}(M) \rar["(c_\Gamma)_\ast"] \dar["\tau_{M,N}"] \& \Omega_*^{\spin}(\Bfree \Gamma) \rar \dar["\tau_{\Gamma,\pi}"] \& \KO_\ast(\Bfree \Gamma) \dar["\tau_{\Gamma,\pi}"] \rar
      \& \KO^\Gamma_{\ast}(\Eub \Gamma) \rar \dar["\underline{\tau}_{\Gamma,\pi}"] \& \KO_\ast(\Cstar \Gamma) \dar["\sigma_{\Gamma,\pi}"] \\
    \Omega_{\ast-k}^{\spin}(N) \rar["(c_\pi)_\ast"] \& \Omega_{\ast-k}^{\spin}(\Bfree \pi) \rar \& \KO_{\ast-k}(\Bfree \pi)\rar
        \& \KO^\pi_{\ast-k}(\Eub \pi) \rar \& \KO_{\ast-k}(\Cstar \pi),
  \end{tikzcd}
    }
  \end{equation}
such that $\tau_{M,N}$ maps the fundamental class of $M$ to that of $N$. The existence of the complete diagram would make it most transparent how $\alpha(N)$ is an obstruction to positive scalar curvature on $M$. But already part of the diagram would suffice to show this under some additional assumptions:

When we use $\Cstar\Gamma=\CstarRed\Gamma$ and $\Gamma$ satisfies the Baum-Connes Conjecture, i.e.\ when $\mu$ is an isomorphism, the rightmost extension can be abstractly defined as $\mu\circ\underline{\tau}_{\Gamma,\pi}\circ\mu^{-1}$. When $\mu$ is at least injective, existence of $\underline{\tau}_{\Gamma,\pi}$ is still enough to conclude that $\alpha(N)\neq 0$ implies $\alpha(M)\neq 0$. Similarly, the existence of $\tau_{\Gamma,\pi}$ suffices when $\Gamma$ satisfies the Strong Novikov Conjecture, i.e.\ when $\mu\circ\iota_*$ is injective.
Note that the Baum-Connes Conjecture for real K-Theory is equivalent to the
one for complex K-Theory, and after inverting 2 this is also true for
injectivity of $\mu$ \cite{schickRealVsComplex}. The map $\iota_*$ is always
rationally injective and an isomorphism if $\Gamma$ is torsion-free.

\medskip

Our first and most complete result deals with codimension one.
Zeidler already proved that if \(k =1\) and the map \(\pi_1(N) \to \pi_1(M)\) is injective, then a suitable transformation on the level of the group \(\Cstar\)-algebras exists~\cite{ZeidlerSubmf}*{Remark 1.8}.
In \cref{sec:TransferKK}
we reprove this and complete the picture to
obtain \cref{eqn:KO-extension-diagram}.
We also generalize the result slightly by dropping the \(\pi_1\)-injectivity requirement and instead work with the image of the induced map \(\pi_1(N) \to \pi_1(M)\).

\begin{theorem}[{see \cref{theo:codim_1_long}}]\label{theo:codim_1_short}
  Let $M$ be a closed connected manifold and $N \subset M$ a closed connected submanifold of codimension one with trivial normal bundle.
  Let \(\Gamma \coloneqq \pi_1(M)\), and let \(\pi \leq \Gamma\) be the image of the homomorphism \(\pi_1(N)\to \pi_1(M) = \Gamma\) which is induced by the inclusion \(N \hookrightarrow M\).
Then the classical transfer map can be extended as in
\cref{eqn:KO-extension-diagram}. This holds for $\Cstar$ either $\CstarRed$ or
$\CstarMax$.

  In particular, if \(M\) is spin and \(\alpha(N) \neq 0 \in \KO_{m-k}(\Cstar
  \pi)\), then \(\alpha(M) \neq 0 \in \KO_{m}(\Cstar \Gamma)\) and \(M\) does
  not admit a metric of positive scalar curvature. 
\end{theorem}
The ultimate goal would be to
obtain \cref{eqn:KO-extension-diagram}
also for submanifolds of higher codimension (and suitable assumptions on the homotopy type).
Unfortunately, in complete generality we do not achieve this.
However, we obtain various
partial results for different generalized homology theories.

\medskip

The obvious choice for $\tau_{M,N}$ is the \emph{classical transfer map}, which can be defined for generalized homology theories:
  Let \(E\) be some multiplicative cohomology theory and suppose that
both $M$ and $N$ admit an $E$-orientation. Then these orientations induce a Thom class $\theta \in E^k(\disk \nu, \sphere \nu)$, where $\disk\nu$ and $\sphere\nu$ denote the disk- and sphere bundle of the normal bundle for $N\subset M$. Let
   \[\tau_{M,N} = \tau_\theta\colon\myE{\ast}(M)\xrightarrow{}\myE{\ast}(M,M\setminus \disk\nu)\iso\myE{\ast}(\disk\nu,\sphere\nu)\xrightarrow{\cap\theta}\myE{\ast-k}(N).\]

In \cref{sec:generalTransfer}, which is partly based on the first-named author's PhD thesis \cite{NitscheThesis}, we introduce for an equivariant generalized cohomology theory $E$ the notion of an \emph{lf-restriction} operation, and we show that singular homology, K-theory and the bordism theories all allow such an operation.
Building on this, we construct transfer maps in the equivariant setting,
\begin{align*}
& & E_*(\Bfree\Gamma)\cong E_*^\Gamma(\Efree\Gamma)
&\to E_{*-k}^\pi(\Efree\Gamma)\cong E_{*-k}(\Bfree\pi) & &\\
\text{and}& &
E_*^\Gamma(\Eub\Gamma)&\to E_{*-k}^\pi(\Eub\Gamma)\cong E_{*-k}^\pi(\Eub\pi),& &
\end{align*}
between the classifying spaces for free and for proper actions.
We obtain a general (but technical) criterion for when this construction can be used to extend the classical transfer map:
\begin{equation}\label[diagram]{eqn:general-extension-diagram}
\begin{tikzcd}
E_*(M)\arrow[r]\arrow[d,"\tau_{M,N}"] &
E_*(\Bfree{\Gamma})\arrow[r]\arrow[d,dashed] &
E^\Gamma_*(\Eub{\Gamma})\arrow[d,dashed]\\
E_{*-k}(N)\arrow[r] &
E_{*-k}(\Bfree{\pi})\arrow[r] &
E^\pi_{*-k}(\Eub{\pi})
\end{tikzcd}
\end{equation}
Our construction is inspired by the ideas of Engel \cite{Engel}, who
constructed an extension of the classical transfer map for singular homology
in the framework of uniformly finite homology. It can be seen as a
generalization and simplification of Engel's arguments. Specializing in
suitable situations to spin
bordism and real K-homology, the method gives rise to
\cref{eqn:KO-extension-diagram} except for the transfer on the level of
$\Cstar$-algebras, which so far has not been constructed in the general setting.

\medskip

In \cref{sec:applications} we apply the general transfer construction to different generalized cohomology theories, always with the goal to obtain information about positive scalar curvature on $M$.

Our only result which yields the desired transformation in all codimensions is for ordinary homology.
It slightly generalizes the result of Engel \cite{Engel}.

\begin{theorem}[see \cref{thm:application-singular-homology}]\label{theo:codim_homol}
  Let $M$ be a closed connected manifold and $N \subset M$ a closed connected submanifold of codimension $k$ with \(A\)-oriented normal bundle for some commutative ring \(A\).
    Let \(\Gamma \coloneqq \pi_1(M)\), and
$\pi={i_N}_*(\pi_1(N))<\pi_1(M)$. Let $E_*=\HZ_*(\blank;A)$ be singular homology with coefficients in $A$.
Assume that
  \begin{enumerate}\setlength\parskip{0pt}\setlength\itemsep{1pt}
  \item if $k>1$, then $\pi_k(N)\to \pi_k(M)$ is surjective;
  \item if \(k>2\), then $\pi_j(M)=0$ for $j=2,\dots,k-1$;
  \item there exists a subset \(S \subseteq \pi_k(N)\) which generates
    \(\pi_k(M)\) such that the composition of the Hurewicz homomorphism
    $\pi_k(N)\to \HZ_k(N;\Z)$ with the evaluation on the Euler class of $\nu$ vanishes on \(S\).
In particular, this is satisfied if the normal bundle is trivial or \(\pi_k(M)=0\) or \(\dim N < \frac{\dim M}{2}\).
  \end{enumerate}
Then the classical transfer map can be extended to $\HZ_*(\Bfree\Gamma;A)\to\HZ_{*-k}(\Bfree\pi;A)$.
\end{theorem}

For the \(\pi_1\)-injective case this result appeared in the article of Engel~\cite{Engel}*{Main Theorem} but without the condition \labelcref{item:EulerClassCondition}.
However, the counterexample described in \cref{ex:EulerClassCondition} below shows that this condition is indeed necessary —even in the \(\pi_1\)-injective case.
Engel informed us that his proof nonetheless works under the additional assumption that \(\pi_k(M)\) is zero or that the normal bundle is trivial, see the updated article on the arXiv~\cite{EngelArXiv}.
These hypotheses are both covered by our condition~\labelcref{item:EulerClassCondition}.

\cref{theo:codim_homol}
already has interesting consequences for the positive scalar curvature problem provided that the rational analytic Novikov conjecture holds for $\Gamma$.
The following applications were also observed by Engel~\cite{Engel}*{p.~425}.

\begin{corollary}[see \cref{cor:application-singular-homology}]\label{cor:homol_KOessentialness}
  Assume that we are in the situation of \cref{theo:codim_homol} \parensup{with \(A=\Z\)}, that the assembly map \(\mu \otimes \Q \colon \KO_\ast(\Bfree \Gamma) \otimes \Q \to \KO_\ast(\Cstar \Gamma) \otimes \Q\) is injective, and that $M$ is spin. Moreover, assume that either
\begin{enumerate}\setlength\parskip{0pt}\setlength\itemsep{0pt}
\item The push forward of the homological fundamental class of \(N\) does not vanish in \(\HZ_\ast(\Bfree \pi; \Q)\), or
\item the normal bundle of $N\subset M$ is trivial and the push forward of the \textup{KO}-fundamental class of \(N\) \parensup{with the induced spin structure} in \(\KO_\ast(\Bfree \pi) \otimes \Q\) does not vanish.
\footnote{This can be rephrased as the condition that there exists a non-vanishing higher \(\hat{\mathrm{A}}\)-genus of \(N\) which —via the map \(\pi_1 N \twoheadrightarrow \pi\)— comes from a cohomology class of \(\pi\).}
\end{enumerate}
Then $\alpha(M)\neq 0$ and hence \(M\) does not admit a metric of positive scalar curvature.
\end{corollary}

For codimension two or three, we also have results for certain generalized cohomology theories.
The crucial observation which makes the next theorem work is that an orientation on a two-dimensional real vector bundle defines a complex structure on it.
This, in turn, induces an orientation in every complex oriented cohomology theory such as \(\spin^{\mathrm{c}}\) bordism and complex \(\K\)-theory (and up to inverting \(2\), also spin bordism and real K-theory).
\begin{theorem}[{see~\cref{thm:application-oriented-codim-2,thm:application-oriented-codim-3}}]\label{thm:application-complexted-oriented-short}
  Let $M$ be a closed connected manifold and $N \subset M$ a closed connected submanifold of codimension \(k \in \{2,3\}\) with oriented normal bundle \(\nu\).
    Let \(\Gamma \coloneqq \pi_1(M)\), and
$\pi\coloneqq{i_N}_*(\pi_1(N))<\pi_1(M)$. Moreover, assume that
\begin{itemize} 
\item
For $k=2$\textup{:} The induced map \(\pi_2(N) \to \pi_2(M)\) is surjective. Also, there exists a subset \(S \subseteq \pi_2(N)\) which generates \(\pi_2(M)\) such that the composition of the Hurewicz homomorphism $\pi_2(N)\to \HZ_2(N;\Z)$ with the first Chern class of $\nu$ vanishes on \(S\), where we use the orientation to view \(\nu\) as a complex line bundle.
\item
For $k=3$\textup{:} The induced map \(\pi_1(N) \to \pi\) is an isomorphism,
\(\pi_2(N) = 0 = \pi_2(M)\) and \(\pi_3(N) \to \pi_3(M)\) is surjective. Also,
$\nu$ is trivial and the Thom class coincides with the pullback of the three
times suspended unit $x\in E^3(\disk^3,\sphere^2)$ under the trivialization map.
\end{itemize}
  Then we have a commutative diagram
  \begin{equation*}
    \begin{tikzcd}
      \Omega^{\spin^{\mathrm{c}}}_\ast(M) \dar{\tau_{M,N}} \rar
        & \Omega^{\spinC}_\ast(\Bfree \Gamma) \dar{\tau_{\Gamma,\pi}} \rar
        & \K_\ast(\Bfree \Gamma) \dar{\tau_{\Gamma,\pi}}\\
      \Omega^{\spin^{\mathrm{c}}}_{\ast-k}(N) \rar
        & \Omega^{\spinC}_{\ast-k}(\Bfree \pi) \rar
        & \K_{\ast-k}(\Bfree \pi).
    \end{tikzcd}
  \end{equation*}
  Moreover, we obtain the analogous diagram for \(\spin\)-bordism and \(\KO\)-theory up to inverting \(2\).
\end{theorem}

In \cref{subsec:extension-to-eub}, which is again based on \cite{NitscheThesis}, we consider once more the most general case of a generalized cohomology theory. We give a sufficient condition for extending the transfer map to the classifying space for proper actions in terms of group cohomology.
Except for very low codimension, however, the condition depends on the group $\pi_1(M)$. The best result that we get independent of $\pi_1(M)$ is in codimension $2$ when the normal bundle of $N\subset M$ is trivial. Specialized to spin bordism and real K-Theory, \cref{thm:application-trivialized-codim-2} gives:

\begin{theorem}\label{thm:trivialized-codim-2-short}
Let $M$ be a closed connected spin manifold and $N \subset M$ a closed connected submanifold of codimension $k=2$ with trivial normal bundle.
Let \(\Gamma \coloneqq \pi_1(M)\) and $\pi\coloneqq \pi_1(N)$.
Moreover, assume that $\pi_1(N)\to\pi_1(M)$ is injective and that
$\pi_2(N)\to\pi_2(M)$ is surjective.

Then we have the following commutative diagram:
\begin{equation*}
\resizebox{\linewidth}{!}{\begin{tikzcd}[ampersand replacement=\&]
\Omega_*^{\spin}(M) \rar \dar["\tau_{M,N}"] \&
\Omega_*^{\spin}(\Bfree \Gamma) \rar \dar["\tau_{\Gamma,\pi}"] \&
\KO_\ast(\Bfree \Gamma) \dar["\tau_{\Gamma,\pi}"] \rar \&
\KO^\Gamma_{\ast}(\Eub \Gamma) \rar \dar["\underline{\tau}_{\Gamma,\pi}"] \&
\KO_\ast(\Cstar \Gamma) \\
\Omega_{\ast-k}^{\spin}(N) \rar \&
\Omega_{\ast-k}^{\spin}(\Bfree \pi) \rar \&
\KO_{\ast-k}(\Bfree \pi)\rar \&
\KO^\pi_{\ast-k}(\Eub \pi) \rar \&
\KO_{\ast-k}(\Cstar \pi)
\end{tikzcd}}
\end{equation*}
\end{theorem}

Note that the assumptions in the preceding theorem precisely match those of \cref{theo:obstruc_codim_2}. It seems very likely that the resulting transfer maps are compatible with Kubota's map on the K-Theory of the maximal group $\Cstar$-algebras $\KO_*(\CstarMax\Gamma)\to\KO_{*-2}(\CstarMax\pi)$ \cite{Kubota}*{proof of Theorem 6.11}, since both Kubota's and our maps arise from a quite \enquote{natural} construction.
The problem with comparing these maps is that our general transfer construction is defined in a very geometric way, and in a setting that is much more general than K-theory. Kubota's construction, on the other hand, relies on tools involving operator algebras such as K-theory with coefficients in the Calkin algebra.

Another obvious line of inquiry would be to further relax the conditions under which the transfer map can be extended in various generalized cohomology theories.
We note, however, that some assumptions on the homotopy type of $M$ and $N$ and on the generalized cohomology theory will always be necessary.

One cannot expect, for example, that the mere existence of a submanifold inclusion was enough to obtain the transfer on group homology. This is only conceivable if the manifolds resemble to a suitable degree the classifying spaces of the groups.
Furthermore,
we do not expect a version of \cref{thm:application-singular-homology} for arbitrary cohomology theories and arbitrary prescribed orientations on the normal bundle.
A —perhaps artificial— counterexample in \(\KO\)-theory is an embedding \(\sphere^1 \subset \sphere^{k+1}\), where we endow the normal bundle with the \(\KO\)-orientation which induces the \enquote{antiperiodic} spin structure on \(\sphere^1\).
However, it is not clear where the precise frontier lies, and hence it would be of interest to find more (counter-)examples.

\section{Geometric setup}\label{sec:setupGeom}
In this short section, we describe the general geometric setup that will be used in the technical sections of this article.
\begin{setup}\label{setup:geom}
Fix a codimension \(k \in \N\).
Let \(M\) and \(N\) be finite CW-complexes and let \(\nu\) be a real vector bundle of fiber dimension \(k\) over \(N\), and let there be an embedding \(\tn \colon \nu \hookrightarrow M\) of the total space of \(\nu\) as an open subset of \(M\).
Usually, we will use the map \(\tn\) implicitly and consider \(\nu\) as a subset of \(M\).

Let \(\sphere \nu \subset \disk \nu \subset \nu \subset M\) be the unit sphere and disk bundles associated to \(\nu\), respectively.
Fix a base-point $x_0 \in N \subset \nu \subset M$.
Let $\Gamma \coloneqq \pi_1(M, x_0)$ and let \(\pi \leq \Gamma\) be the image of the map \(\pi_1(N, x_0) \to \pi_1(M, x_0)\) induced by the restriction of the embedding \(\nu \hookrightarrow M\) to the zero section.
Let $\tilde{p} \colon \ucov{M} \to M$ denote the universal covering of \(M\).
Let $\bar{M} = \pi \backslash \ucov{M}$, and $\underline{p} \colon \ucov{M}
\to \bar{M}$ and $p \colon \bar{M} \to M$ be the corresponding covering maps.
Fix a base-point $\tilde{x}_0 \in \tilde{p}^{-1}(x_0)$ and let $\bar{x}_0 = \underline{p}(\tilde{x}_0) \in \bar{M}$.
Then the embedding \(\tn \colon \nu \hookrightarrow M\) uniquely lifts to an embedding \(\bar{\tn} \colon \nu \hookrightarrow \bar{M}\) taking \(x_0\) to \(\bar{x}_0\).
Then we let \(\ucov{\disk\nu} \to \disk \nu\) be the restriction of the covering \(\ucov{M} \to \bar{M}\) via \(\bar{\tn} \colon \disk \nu \hookrightarrow \bar{M}\).
We obtain an embedding \(\tilde{\tn} \colon \ucov{\disk\nu} \hookrightarrow \ucov{M}\) as a \(\pi\)-compact subset.
Similarly, define \(\ucov{\sphere \nu}\) and \(\ucov{N}\) to be the restrictions to \(\sphere \nu\) and \(N\), respectively.
\end{setup}

\begin{remark}
  The covering \(\ucov{N} \to N\) is the connected normal covering of \(N\) with \(\pi_1 \ucov{N} = \ker(\pi_1 N \twoheadrightarrow \pi)\).
  In particular, if the map \(\pi_1 N \to \pi\) is an isomorphism, then \(\ucov{N}\) is the universal covering of \(N\).
\end{remark}

\begin{remark}
  The typical situation we consider is that \(N \subset M\) is a submanifold embedding of codimension \(k\) with normal bundle \(\nu\) which is embedded into \(M\) as a tubular neighborhood of \(N\).
\end{remark}

\section{Codimension one transfer via KK-theory}\label{sec:TransferKK}
In this section, we state and prove the precise
version of Theorem \ref{theo:codim_1_short}.

\begin{theorem}\label{theo:codim_1_long}
  Let $M$ be a closed spin manifold and $N\subset M$ a submanifold of
  codimension $1$ with trivial normal bundle, both connected. Set
  $\pi:=\im(\pi_1(N)\to
  \pi_1(M)=\Gamma)$.
  Then the dashed homomorphisms in the following diagram exist and make it commutative.
  \begin{equation}\label{eq:codim1Transfer}
  \resizebox{0.9 \linewidth}{!}{
  \begin{tikzcd}[ampersand replacement=\&]
    \Omega_*^{\spin}(M) \rar["c_\ast"] \dar["\tau_{M,N}"] \& \Omega_*^{\spin}(\Bfree \Gamma) \rar["\alpha"] \dar[dashed,"\tau_{\Gamma, \pi}"] \& \KO_\ast(\Bfree \Gamma) \dar[dashed,"\tau_{\Gamma, \pi}"] \rar["t"]
      \& \KO^\Gamma_{\ast}(\Eub \Gamma) \rar["\mu^\Gamma"] \dar[dashed,"\underline{\tau}_{\Gamma, \pi}"] \& \KO_\ast(\Cstar \Gamma) \dar[dashed,"\sigma_{\Gamma, \tau}"] \\
    \Omega_{\ast-1}^{\spin}(N) \rar["\bar{c}_\ast"] \& \Omega_{\ast-1}^{\spin}(\Bfree \pi) \rar["\alpha"] \& \KO_{\ast-1}(\Bfree \pi)\rar["t"]
        \& \KO^\pi_{\ast-1}(\Eub \pi) \rar["\mu^\pi"] \& \KO_{\ast-1}(\Cstar \pi),
  \end{tikzcd}
    }
  \end{equation}
    The remaining arrows are defined as follows.
    The map $\tau_{M, N}$ on the left-hand side is induced by the Pontryagin--Thom collapse $M \to \Sigma N$ \parensup{via the trivial normal bundle}.
  We use the map $c\colon M\to \Bfree \Gamma$ inducing the identity
  on fundamental groups, and $\bar{c} \colon N \to \Bfree \pi$ the composition of $N \to \Bfree\pi_1(N)$ which is
  the identity on the fundamental group and
  $\Bfree \pi_1(N)\to \Bfree\pi$ induced by $\pi_1(N)
  \onto \pi$.
  The map $\alpha$ is the Atiyah orientation from spin bordism to real
  \(\K\)-homology.
  Furthermore, $t$ is the composition of the transfer isomorphism
  $\KO_*(\Bfree\Gamma)\to \KO_*^\Gamma(\Efree\Gamma)$ with the map induced by the unique
 homotopy class of  $\Gamma$-equivariant maps
  $\Efree\Gamma\to\Eub\Gamma$ \parensup{using that $\Eub\Gamma$ is a
  universal space for proper $\Gamma$-actions}.
  Finally, $\mu^\Gamma$ and $\mu^\pi$ are the Baum--Connes assembly maps
  corresponding to $\Gamma$ and $\pi$, respectively.
\end{theorem}

  Our definitions of the maps \(\underline{\tau}_{\Gamma, \pi}\) and \(\sigma_{\Gamma, \pi}\) are based on a direct construction of a suitable boundary class in equivariant KK-theory similar to the work of Oyono–Oyono~\cite{oyonooyono} on groups acting on trees.
  This is not surprising because the situation of \cref{theo:codim_1_long} implies that \(\Gamma\) is an amalgamated product or an HNN-extension.
  The previous construction of Zeidler~\cite{ZeidlerSubmf}*{Remark 1.8} on the
  level of the group \(\Cstar\)-algebra is based on a different method: it uses coarse index theory and the partitioned manifold index theorem.
  We do not prove that these two approaches yield the same map but nevertheless strongly expect that to be the case.

To prove the theorem, we consider a slightly more abstract setup based on \cref{setup:geom}.

\begin{setup} \label{setup:KKTransfer}
Suppose that we are in the situation of \cref{setup:geom} with \(k = 1\) and where \(\nu\) is the trivial line bundle.
In this section, we will keep the embedding \(\tn \colon N \times \R \hookrightarrow M\) explicit in our notation.
Hence the chosen base-point \(x_0 \in M\) is of the form \(x_0 = \tn(y_0, 0)\) for some \(y_0 \in N\).
Since $\pi_1(\bar{M}, \bar{x}_0) = \pi = \tn_{\ast}(\pi_1 (N \times \{0\}, y_0))$, codimension one and a Mayer--Vietoris argument imply that $\bar{M} \setminus \bar{\tn}(N \times \{0\})$ consists of two connected components.
We denote the closures of these components by $\bar{M}_\pm$, where the sign in the notation is determined by the requirement $\bar{\tn}(N \times \{\pm 1 \}) \subseteq \bar{M}_\pm$.
\end{setup}

We assume \cref{setup:KKTransfer} in the following.
Fix an odd function $\chi_0 \colon \R \to [-1,1]$ with $\chi_0(\pm t) = \pm 1$ for $t \geqslant 1$.
Then we define a function $\chi \colon \bar{M} \to [-1,1]$ by
\[
  \chi(\bar{x}) = \begin{cases}
                   \chi_0(t) & \bar{x} = \bar{\tn}(y, t) \in \bar{\tn}(N \times [-1,1]), \\
                   \pm 1 & \bar{x} \in M_\pm \setminus \bar{\tn}(N \times [-1,1]).
                  \end{cases}.
\]
Let $\tilde{\chi} \colon \ucov{M} \to [-1,1]$ denote the pullback of $\chi$ to $\ucov{M}$.

\begin{lemma}\label{lem:finitelyManyClasse}
  Let $K \subseteq \ucov{M}$ be a compact subset.
  Then there are only finitely many different cosets $\gamma \pi \in \Gamma / \pi$ such that
  \[ \sup_{\tilde{x}_1, \tilde{x}_2 \in K} \left| \tilde{\chi}(\gamma^{-1} \tilde{x}_1 ) - \tilde{\chi}(\gamma^{-1} \tilde{x}_2) \right| \neq 0.\]
\end{lemma}
\begin{proof}
  Since \(\ucov{M}\) is a path-connected CW complex, every compact subset is contained in a path-connected compact subset.
  Hence we may assume without loss of generality that \(K\) is path-connected.
  By construction, the set \(\tilde{\tn}(\ucov{N} \times [-1,1])\) is \(\pi\)-compact.
  Thus, since the action of \(\Gamma\) on \(\ucov{M}\) is proper, there are only finitely many \(\gamma \pi \in \Gamma / \pi\) such that \(\gamma^{-1} K \cap \tilde{\tn}(\ucov{N} \times [-1,1])\) is non-empty.
  The set \(\ucov{M} \setminus \tilde{\tn}(\ucov{N} \times [-1,1])\) consists of two connected components on each of which the function \(\tilde{\chi}\) is constant.
  Using that \(K\) is connected, we conclude that for all but finitely many cosets \(\gamma \pi \in \Gamma / \pi\), the set \(\gamma^{-1} K\) is contained in one of the components. Hence the function \(\tilde{\chi}\) is constant on \(\gamma^{-1} K\) for all but finitely \(\gamma \pi \in \Gamma / \pi\).
  This proves the lemma.
\end{proof}

In the following, we work with equivariant KK-theory for Real $\Cstar$-algebras.
We use the following picture of $\KK_{-1}^\Gamma(A, B)$ for separable Real $\Gamma$-$\Cstar$-algebras $A$ and $B$ (compare~\cite{blackadar}*{20.2 and 17.5.2}):
elements are represented by triples $(E, \phi, T)$, where $E$ is a countably generated Real Hilbert $B$-module with an action of $\Gamma$, $\phi$ is an equivariant representation of $A$ on $E$ by adjointable operators, $T$ is a self-adjoint operator on $E$ such that $[\phi(a), T]$, $\phi(a)(T^2 - 1)$ and $\phi(a) (g \cdot F - F)$ are compact operators (in the Hilbert $B$-module sense) for all $a \in A$ and $g \in \Gamma$.

The collapse induces a map $\partial_{\ucov{M}}^\pi \colon \KO^{\lf, \pi}_\ast(\ucov{M}) \xrightarrow{\tn^!} \KO_\ast^{\lf, \pi}(\ucov{N} \times (-1,1)) \cong \KO^\pi_{\ast-1}(\ucov{N})$.
We view it as an element of $\KK^\pi_{-1}(\Cz(\ucov{N}), \Cz(\ucov{M}))$ which is explicitly given as follows:
Let $T_0 \coloneqq 1_{\ucov{N}} \otimes \chi_0$.
This represents a $\pi$-equivariant self-adjoint multiplier of $\Cz(\ucov{N} \times (-1,1))$ with $f \otimes 1\ (T_0^2 - 1) \in  \Cz(\ucov{N} \times (-1,1))$ for all $f \in \Cz(\ucov{N})$.
Thus $T_0$ represents a class $\partial^\pi_{\ucov{N} \times (-1,1)} \in \KK^\pi_{-1}(\Cz(\ucov{N}), \Cz(\ucov{N} \times (-1,1)))$.
Then the desired class is $\partial^\pi_{\ucov{M}} = \tilde{\tn}_! \partial^\pi_{\ucov{N} \times (-1,1)}$.

\begin{remark}\label{rem:PontryaginThomViaKK}
  We have the commutative diagram
  \[
    \begin{tikzcd}
      \KO_\ast(M) \rar["\cong"] \ar[dd,"\tn^!"] \ar[ddd, "\tau_{M,N}"', bend right=90]& \KO^\Gamma_{\ast}(\ucov{M}) \dar["\mathrm{r}_\pi^\Gamma"] \\
      & \KO_\ast^{\lf,\pi}(\ucov{M}) \dar["\tilde{\tn}^!"] \ar[dd, bend left=90,"\partial^\pi_{\ucov{M}}"]\\
      \KO_\ast^\lf(N \times (-1, 1)) \rar["\cong"] \dar["\cong","\partial_{N \times (-1,1)}"'] & \KO^{\lf,\pi}_{\ast}(\ucov{N} \times (-1,1)) \dar["\partial^\pi_{\ucov{N} \times (-1,1)}","\cong"'] \\
      \KO_{\ast-1}(N) \rar["\cong"] & \KO_{\ast-1}^\pi(\ucov{N}),
    \end{tikzcd}
  \]
  where \(\mathrm{r}_\pi^\Gamma\) is the forgetful map restricting equivariance to \(\pi\) and \(\partial_{N \times (-1,1)}\) is defined similarly as \(\partial^\pi_{\ucov{N} \times (-1,1)}\).
  The vertical composition on the left-hand side is by definition the map \(\tau_{M,N}\) which is induced by the Pontryagin--Thom collapse.
\end{remark}

Next we discuss induction (compare~\cite{KasparovConspectus}*{\S 5}).
Let $A$ be a $\pi$-$\Cstar$-algebra.
Then the induced $\Gamma$-$\Cstar$-algebra $\Ind_\pi^\Gamma A$ consists of all those functions $f \in \Cb(\Gamma, A)$ such that $f(\gamma h) = h^{-1} f(\gamma)$ for all $\gamma \in \Gamma$, $h \in \pi$ and the function $\Gamma/\pi \to \R$, $\gamma \pi \mapsto \|f(\gamma)\|$ vanishes at infinity.
For two $\pi$-$\Cstar$-algebras there is Kasparov's induction homomorphisms
\[
  \ii_\pi^\Gamma \colon \KK^\pi(A, B) \to \KK^\Gamma(\Ind_{\pi}^\Gamma A, \Ind_{\pi}^\Gamma B).
\]
Note that if $A = \Cz(Z)$ for some proper locally compact $\pi$-space $Z$, then $\Ind_\pi^\Gamma A = \Cz(\Gamma \times_\pi Z)$.
If $Z$ is a $\pi$-invariant subspace of a proper locally compact $\Gamma$-space $W$, then there is a proper $\Gamma$-equivariant map $\Gamma \times_\pi Z \to W$ which takes $[\gamma, z]$ to $\gamma \cdot z \in W$.
Composing this map with induction yields a map
\begin{align}
  \Ind_\pi^\Gamma \colon \KK^\pi(\Cz(Z), B) &\xrightarrow{\ii_\pi^\Gamma} \KK^\Gamma(\Cz(\Gamma \times_\pi Z), \Ind_\pi^\Gamma B) \nonumber \\
  &\rightarrow \KK^\Gamma(\Cz(W), \Ind_\pi^\Gamma B) \label{eq:inductionCompose}
\end{align}
Passing to the colimit over all proper $\pi$-invariant $\pi$-compact subsets
$Z$ of a universal space for proper actions $\Eub \Gamma$ (which is also an
$\Eub\pi$ with the restricted action) yields the induction map
\begin{equation}
  \Ind_\pi^\Gamma \colon \KO_\ast^\pi(\Eub \Gamma; B) \to \KO_\ast^\Gamma(\Eub \Gamma; \Ind_\pi^\Gamma B)\label{eq:inductionIso}
\end{equation}

\begin{theorem}[\cites{chabertEchterhoff,oyonooyono}]\label{thm:KKinductionIso}
  The induction map \labelcref{eq:inductionIso} is an isomorphism.
\end{theorem}
\begin{proof}
  For \emph{complex} K-homology, this result is due to Oyono--Oyono~\cite{oyonooyono}.
  See also Chabert--Echterhoff~\cite{chabertEchterhoff} for a generalization to the case of locally compact groups.
  The real case can be obtained from the complex case via the technique of~\cite{schickRealVsComplex}.
\end{proof}

After this interlude on induction we return to our setup:

\begin{definition}\label{defi:GroupBoundaryClass}
  Assume \cref{setup:KKTransfer}.
  We let $\partial_{\Gamma/\pi} \in \KK^{\Gamma}(\R, \Cz(\Gamma/\pi))$ be the class defined as follows.
  Let $i \colon \Gamma / \pi \hookrightarrow \bar{M}$, $i(\gamma \pi) := \underline{p}(\gamma^{-1} \tilde{x}_0) \in \bar{M} = \pi \backslash \ucov{M}$.
  Set $\chi_{\Gamma / \pi} \colon \Gamma / \pi \to [-1,1]$, $\chi_{\Gamma / \pi}(z) = \chi(i(z))$.
  Then $\chi_{\Gamma / \pi}$ is a self-adjoint multiplier of $\Cz(\Gamma / \pi)$ which satisfies $\chi_{\Gamma/\pi}^2 = 1$ modulo \(\Cz(\Gamma/\pi)\).
  Furthermore, $\chi_{\Gamma / \pi}$ is $\Gamma$-invariant modulo $\Cz(\Gamma/\pi)$:
  For each fixed $g \in \Gamma$, \cref{lem:finitelyManyClasse} with $K \coloneqq \{\tilde{x}_0, g \cdot \tilde{x}_0 \}$ implies that the  function $z \mapsto \chi_{\Gamma / \pi}(z) - \chi_{\Gamma / \pi}(g^{-1} \cdot z)$ has finite support.
  Thus $\chi_{\Gamma / \pi}$ represents a class $\partial_{\Gamma / \pi} \in \KK_{-1}^\Gamma(\R, \Cz(\Gamma/\pi))$.
\end{definition}

If $W$ is a proper $\Gamma$-space, then we in fact have a $\Gamma$-equivariant homeomorphism
\[\Gamma/\pi \times W \cong \Gamma \times_{\pi} W, \quad (\gamma \pi, x) \mapsto [\gamma, \gamma^{-1} x].\]
In particular, in \cref{setup:KKTransfer}, there is a canonical isomorphism of $\Gamma$-\(\Cstar\)-algebras
\begin{equation}
  \Ind_\pi^\Gamma(\Cz(\ucov{M})) = \Cz(\Gamma \times_\pi \ucov{M}) \cong \Cz(\Gamma / \pi \times \ucov{M}) = \Cz(\Gamma/\pi) \otimes \Cz(\ucov{M}).\label{eq:Xinduction}
  \end{equation}

\begin{proposition}\label{prop:boundaryElement}
Assume \cref{setup:KKTransfer}.
  Let $\partial_{\Gamma / \pi} \in \KK_{-1}^\Gamma(\R, \Cz(\Gamma / \pi))$ be
  as in \cref{defi:GroupBoundaryClass}.
  Then
  \[
    \Ind_\pi^\Gamma(\partial^\pi_{\ucov{M}}) = \partial_{\Gamma / \pi} \otimes_\R 1_{\Cz(\ucov{M})} \in \KK_{-1}^\Gamma(\Cz(\ucov{M}), \Cz(\Gamma / \pi) \otimes \Cz(\ucov{M})).
  \]
  This identity implicitly uses the canonical isomorphism~\labelcref{eq:Xinduction} applied to the right-hand side argument of \(\KK\).
\end{proposition}
\begin{proof}
Recall that $\partial_{\ucov{M}}^\pi$ is represented by the operator $T_0 \coloneqq 1_{\ucov{N}} \otimes \chi_0$ acting as a multiplier of $\Cz(\ucov{N} \times (-1,1))$.
Then $\Ind_{\pi}^\Gamma \partial_{\ucov{M}}^\pi \in \KK_{-1}^\Gamma(\Cz(\ucov{M}), \Ind_\pi^\Gamma \Cz(\ucov{M}))$ is represented by the triple
\[ \Xi \coloneqq (\Cz(\Gamma \times_\pi \ucov{N} \times(-1,1)), \phi, 1_{\Gamma \times_\pi \ucov{N}} \otimes \chi_0), \]
 where we view $\Cz(\Gamma \times_\pi \ucov{N} \times(-1,1))$ as a Hilbert $\Cz(\Gamma \times_\pi \ucov{M})$-module via the inclusion $\Gamma \times_\pi \tn_! \colon \Cz(\Gamma \times_\pi \ucov{N} \times (-1,1)) \hookrightarrow \Cz(\Gamma \times_\pi \ucov{M})$ and $\phi$ is the multiplication representation given as the composition
 \[ \phi \colon \Cz(\ucov{M}) \to \Cz(\Gamma \times_\pi \ucov{N}) \to \Cb(\Gamma \times_\pi \ucov{N} \times (-1,1)),\]
 where the first map is induced by $\Gamma \times_\pi \ucov{N} \to \ucov{M}$, $(\gamma, y) \mapsto \gamma \cdot \tn(\tilde{y}, 0)$.

 We first observe that we may change the representation $\phi$ by a straightforward homotopy to
 \[
 \phi^\prime \colon \Cz(\ucov{M}) \xrightarrow{\psi} \Cb(\Gamma \times_\pi \ucov{M}) \xrightarrow{(\Gamma \times_\pi \tn)^\ast \otimes 1} \Cb(\Gamma \times_\pi \ucov{N} \times (-1,1)),
 \]
  where the first map is induced by $\Gamma \times_\pi \ucov{M} \to \ucov{M}$, $(\gamma, x) \mapsto \gamma \cdot x$. 
  Then the triple
 \[ \Xi^\prime \coloneqq (\Cz(\Gamma \times_\pi \ucov{N} \times(-1,1)), \phi^\prime, 1_{\Gamma \times_\pi \ucov{N}} \otimes \chi_0), \]
 represents the same class as $\Xi$.
 Next the triple $\Xi^\prime$ may be extended to
 \[
 \Xi^{\prime \prime} \coloneqq (\Cz(\Gamma \times_\pi \ucov{M}), \psi, 1_\Gamma \otimes \tilde{\chi}),
 \]
 where  $(1_\Gamma \otimes \tilde{\chi})(\gamma, x) = \tilde{\chi}(x)$.
 This still represents the same KK-class because $\tilde{\chi}$ agrees with  $1 \otimes \chi_0$ on $\ucov{N} \times (-1,1)$ and is invertible outside.
 Finally, applying the homeomorphism $\Gamma \times_\pi \ucov{M} \cong \Gamma / \pi \times \ucov{M}$, we see that $\Xi^{\prime \prime}$ is isomorphic to
 \[ \Xi^{\prime \prime \prime} \coloneqq (\Cz(\Gamma / \pi) \otimes \Cz( \ucov{M}), 1 \otimes \mu, \omega), \]
 where $\omega \colon \Gamma / \pi \times \ucov{M} \to [-1,1]$, $\omega(\gamma \pi, x) = \tilde{\chi}(\gamma^{-1} x)$ and $\mu$ is simply the multiplication representation of $\Cz(\ucov{M})$ on itself.
 To be more precise, the isomorphism between \(\Xi^{\prime \prime}\) and \(\Xi^{\prime \prime \prime}\) is an isomorphism of Kasparov modules over the isomorphism of \(\Gamma\)-\(\Cstar\)-algebras \labelcref{eq:Xinduction} in the right-hand side argument.

 The element $\partial_{\Gamma / \pi} \otimes_\R 1_{\Cz(\ucov{M})}$ is represented by the triple
 \[\Upsilon \coloneqq (\Cz(\Gamma / \pi) \otimes \Cz( \ucov{M}), 1 \otimes \mu, \chi_{\Gamma / \pi}\otimes 1). \]
 We claim that $\Xi^{\prime \prime \prime} $ and $\Upsilon$ represent the same KK-class.
 Indeed, for each compactly supported element $f \in \Cz(\ucov{M})$,  \cref{lem:finitelyManyClasse} implies that
 \[ (\gamma \pi, x) \mapsto f(x) ( \tilde{\chi}  (\gamma^{-1} \tilde{x}_0) - \tilde{\chi}(\gamma^{-1} x)) \]
 is compactly supported on $\Gamma / \pi \times \ucov{M}$.
 Thus $1 \otimes \mu(f) (1 \otimes \chi_{\Gamma / \pi} - \omega) \in \Cz(\Gamma / \pi) \otimes \Cz(\ucov{M})$ for each $f \in \Cz(\ucov{M})$.
 This proves the claim and hence finishes the proof of the proposition.
\end{proof}

\begin{lemma}\label{lem:descentCommutative}
  Let $\pi \leq \Gamma$ be a subgroup of a discrete group $\Gamma$.
  Let $\xi \in \KK^\Gamma_{-k}(\R, \Cz(\Gamma / \pi))$.
  Then we have a commutative diagram,
  \begin{equation*}
    \begin{tikzcd}
      \KO^\Gamma_\ast(\Eub \Gamma) \rar{\mu^\Gamma} \dar{\xi}
      & \KO_\ast(\Cstar \Gamma) \dar{j_\Gamma(\xi)} \\
      \KO^\Gamma_{\ast - k}(\Eub \Gamma, \Cz(\Gamma / \pi)) \rar{\mu^\Gamma_{\Cz(\Gamma / \pi)}}
      & \KO_{\ast-k}(\Cz(\Gamma / \pi) \rtimes \Gamma) \\
      \KO^\pi_{\ast - k}(\Eub \Gamma) \rar{\mu^\pi} \uar{\Ind_{\pi}^\Gamma} &  \KO_{\ast -k }(\Cstar \pi) \uar{\cong}
    \end{tikzcd},
  \end{equation*}
  where $j_\Gamma(\xi) \in \KK_{-k}(\Cstar \Gamma, \Cz(\Gamma / \pi) \rtimes
  \Gamma)$ is the descent homomorphism applied to $\xi$. This holds for either
  reduced group $\Cstar$-algebras and crossed products as well as for maximal
  ones. 
\end{lemma}
\begin{proof}
  The fact that the bottom square in the diagram exists and commutes is a standard fact on the induction isomorphism, see for instance~\cite{chabertEchterhoff}*{Proposition~2.3}.

  To see that the top square commutes, recall the definition of the Baum--Connes assembly map:
  Let $X \subseteq \Eub \Gamma$ be a $\Gamma$-compact subset and $x \in \KK_\ast^\Gamma(\Cz(X), \R)$.
  Then $j_\Gamma(x) \in \KK_\ast(\Cz(X) \rtimes \Gamma, \Cstar \Gamma)$ and $\mu^\Gamma(x) = [p_{X, \Gamma}] \otimes_{\Cz(X) \rtimes \Gamma} j_\Gamma(x)$, where $[p_{X, \Gamma}] \in \KK(\R, \Cz(X) \rtimes \Gamma)$ is the $\K$-theory class determined by a cutoff function for $(X, \Gamma)$.
  Similarly, $\mu^\Gamma_{\Cz(\Gamma/\pi)}(x \otimes_\R \xi) = [p_{X, \Gamma}] \otimes_{\Cz(X) \rtimes \Gamma} j_\Gamma(x \otimes_\R \xi)$.
  Moreover, since Kasparov's descent homomorphism is compatible with the composition product (see~\cite{KasparovConspectus}*{\S 6, Theorem 1}), we have $j_\Gamma(x \otimes_\R \xi) = j_\Gamma(x) \otimes_{\Cstar \Gamma} j_\Gamma(\xi)$.
  Thus
  \begin{align*}
    \mu^\Gamma_{\Cz(\Gamma/\pi)}(x \otimes_\R \xi) &=
      [p_{X, \Gamma}] \otimes_{\Cz(X) \rtimes \Gamma} j_\Gamma(x \otimes_\R \xi) \\
    &= [p_{X, \Gamma}] \otimes_{\Cz(X) \rtimes \Gamma}j_\Gamma(x) \otimes_{\Cstar \Gamma} j_\Gamma(\xi) \\
    &= \mu^\Gamma(x) \otimes_{\Cstar \Gamma} j_\Gamma(\xi). \qedhere
  \end{align*}
\end{proof}


\begin{lemma} \label{lem:inductionDiagram}
Let $\partial_{\Gamma / \pi}$ be as in \cref{prop:boundaryElement}.
Then the following diagram commutes:
\begin{equation*}
  \resizebox{\linewidth}{!}{
\begin{tikzcd}[ampersand replacement=\&]
\KO_\ast(M) \rar{\cong} \ar[dd,"\tau_{M,N}"] \& \KO_\ast^\Gamma(\ucov{M}) \rar[equal] \dar["\mathrm{r}^\Gamma_\pi"] \& \KK_\ast^\Gamma(\Cz(\ucov{M}), \R) \rar \dar["\blank \otimes_\R \partial_{\Gamma/\pi}"] \& \RKO_\ast^\Gamma(\Eub \Gamma) \dar["\blank \otimes_\R \partial_{\Gamma/\pi}"]  \\
 \& \KO_\ast^{\lf, \pi}(\ucov{M})  \dar{\partial^\pi_{\ucov{M}}}  \ar[r, phantom, "\circledast"] \& \KK^\Gamma_{\ast-1}(\Cz(\ucov{M}), \Cz(\Gamma / \pi)) \rar \& \RKO^\Gamma_{\ast-1}(\Eub \Gamma;  \Cz(\Gamma / \pi)) \\
\KO_{\ast-1}(N) \rar["\cong"] \& \KO_{\ast-1}^\pi(\ucov{N}) \rar[equal] \& \KK_{\ast-1}^\pi(\Cz(\ucov{N}), \R) \rar \uar["\Ind_\pi^\Gamma"] \& \RKO_{\ast-1}^{\pi}(\Eub \Gamma) \uar["\Ind_\pi^\Gamma"]
\end{tikzcd}}
\end{equation*}

\end{lemma}
\begin{proof}
  The rectangle on the left-hand side was already discussed in \cref{rem:PontryaginThomViaKK}.
  The commutativity of the squares on the right-hand side is due to naturality of the Kasparov product and the induction map.
  It remains to show that the rectangle $\circledast$ commutes.
  Indeed, let $\xi \in \KK_\ast^\Gamma(\Cz(\ucov{M}), \R)$.
  Then by \cite{KasparovConspectus}*{\S 5, Theorem 1} and \cref{prop:boundaryElement}, we obtain
  \begin{align*}
  \Ind_\pi^\Gamma \left( \partial_{\ucov{M}}^\pi \otimes_{\Cz(\ucov{M})} \mathrm{r}_\pi^\Gamma(\xi) \right)
  &= \Ind_\pi^\Gamma(\partial^\pi_{\ucov{M}}) \otimes_{\Ind_\pi^\Gamma(\Cz(\ucov{M}))} \mathrm{i}_\pi^\Gamma \mathrm{r}_\pi^\Gamma(\xi) \\
  &= \left( \partial_{\Gamma/\pi} \otimes_\R 1_{\Cz(\ucov{M})} \right)
  \otimes_{\Cz(\Gamma / \pi) \otimes \Cz(\ucov{M})}
  \left( 1_{\Cz(\Gamma / \pi)} \otimes_\R \xi \right) \\
  &= \partial_{\Gamma / \pi} \otimes_\R \xi = \xi \otimes_\R \partial_{\Gamma / \pi}.
  \end{align*}
  We have again implicitly used the isomorphism \labelcref{eq:Xinduction}.
\end{proof}

\begin{theorem}\label{thm:generalKKtransfer}
  Suppose \cref{setup:KKTransfer}.
  Then we have a commutative diagram:
  \begin{center}
  \begin{tikzcd}
\KO_\ast(M) \dar \rar \dar & \RKO_\ast^\Gamma(\Eub \Gamma) \rar \dar & \KO_\ast(\Cstar \Gamma) \dar \\
\KO_{\ast-1}(N) \rar & \RKO_{\ast-1}^\pi(\Eub \pi) \rar & \KO_{\ast-1}(\Cstar \pi)
\end{tikzcd}
\end{center}
This  holds for the reduced as well as maximal group $\Cstar$-algebras.
Moreover, this diagram is natural in the data of \cref{setup:KKTransfer} for
fixed subgroup inclusions $\pi \leq \Gamma$.
\end{theorem}
\begin{proof}
  The left square follows from \cref{lem:inductionDiagram} and the fact that the induction map $\Ind_{\pi}^\Gamma \colon \RKO^\pi_\ast(\Eub \Gamma) \to \RKO^\Gamma_\ast(\Eub \Gamma, \Cz(\Gamma / \pi))$ is an isomorphism (see~\cref{thm:KKinductionIso}).
  The right square follows from \cref{lem:descentCommutative}.

  Next we explain naturality.
  Let $\pi \leq \Gamma$ be fixed and let $\tn \colon N \times \R \hookrightarrow M$ and $\tn^\prime \colon N^\prime \times \R \hookrightarrow M^\prime$ both as in \cref{setup:KKTransfer}.
  Suppose that there are continuous maps $f_M \colon M \to M^\prime$ and $f_N \colon N \to N^\prime$ such that $f_M \circ \tn = \tn^\prime \circ (f_N \times \id)$ and $f_M$ induces the identity on \(\pi_1\).
  Then $f_M$ and $f_N$ induce equivariant maps $f_{\ucov{M}}$ and $f_{\ucov{N}}$ between the corresponding covering spaces and hence equivariant $\ast$\nobreakdash-homomorphisms between the associated function algebras.
  By abuse of notation, we also denote the resulting KK-elements by $f_{\ucov{M}} \in \KK^\pi_0(\Cz(\ucov{M}^\prime), \Cz(\ucov{M}))$ and $f_{\ucov{N}} \in \KK^\pi_0(\Cz(\ucov{N}^\prime), \Cz(\ucov{N}))$.
  The fact that the Mayer--Vietoris boundary classes are natural can then be expressed as:
  \begin{equation}
    f_{\ucov{N}} \otimes_{\Cz(\ucov{N})} \partial_{\ucov{M}}^\pi = \partial_{\ucov{M}^\prime}^\pi \otimes_{\Cz(\ucov{M}^\prime)} f_{\ucov{M}}.
    \label{eq:KKMVNatural}
  \end{equation}
  By construction, the functions $\tilde{\chi} \colon \ucov{M} \to \R$ and $\tilde{\chi}^\prime \colon \ucov{M}^\prime \to \R$ satisfy $\tilde{\chi}^\prime \circ f_{\ucov{M}} = \tilde{\chi}$.
  This shows that the element $\partial_{\Gamma/\pi}$ from \cref{defi:GroupBoundaryClass} is in fact the same for $\ucov{M}$ and $\ucov{M}^\prime$.
  This together with \labelcref{eq:KKMVNatural} implies that the commutative diagram is natural in this situation.
\end{proof}

\begin{remark}
  With little extra work, the naturality assertion of Theorem
  \ref{thm:generalKKtransfer} extends to the case where the groups change. We
  leave the details to the reader.
\end{remark}

\begin{proof}[Proof of \cref{theo:codim_1_long}]
  We first observe that the pair $(M, N)$ precisely fits into \cref{setup:KKTransfer}.
  Thus \cref{thm:generalKKtransfer} already yields the resulting diagram except for the two transformations in the middle which involve the spin bordism and $\KO$-homology of $\Bfree \Gamma$.
  To construct these, we will see that $\Bfree \pi$ and $\Bfree \Gamma$ can be constructed in such a way that $\Bfree \pi$ is a subspace of $\Bfree \Gamma$ and \cref{setup:KKTransfer} becomes applicable.
  This uses the special structure of $\Gamma$ in this codimension $1$
  situation.
  We have two cases:

  The first case is that $N$ separates $M$ into two connected
  components.
  That is, $M= M_+\cup_N M_-$, with codimension $0$ submanifolds
  $M_+,M_-$, both with boundary $N$. Set
  \begin{equation*}
\Gamma_+:=\im (\pi_1(M_+)\to
  \pi_1(M))\subset \Gamma; \qquad \Gamma_-:=\im (\pi_1(M_-)\to
  \pi_1(M))\subset \Gamma.
\end{equation*}
By the van Kampen theorem we then have an
  amalgamated free product $\Gamma = \Gamma_+ *_\pi\Gamma_-$.

  The second case is that $N$ is non-separating.
  Then let $W$ be the manifold with boundary obtained by cutting $M$ open along $N$. $W$ has two
  boundary pieces $N_+,N_-$  both identified with $N$, and we have $M= W/\sim$,
  where $\sim$ is the equivalence relation identifying $N_+$ with $N_-$.
  Let $H:=\im(\pi_1(W)\to \pi_1(M))\subset \Gamma$ be the image under the
  collapse map, where we place the basepoint in $N_+$. The inclusion
  $N=N_+\into W$ then induces an inclusion $\pi\into H$ (injective because
  $\pi\subset\Gamma$). Let $\gamma\colon
  [0,1]\to W$ be any   path from the basepoint in $N_+$ to the corresponding
  point in $N_-$; conjugation with it allows to define a second embedding (as
  $\gamma^{-1}\pi\gamma\subset \Gamma$) induced by the inclusion $N=N_-\subset W$. The van
  Kampen theorem now implies that we get $\Gamma$ as an HNN extension $\Gamma = H *_{\pi=\gamma^{-1}\pi\gamma}$, see e.g.\ \cite{SW77TopologicalMethods}*{Proposition~1.2}.

 Then one can construct a model of $\Bfree \Gamma$ out of
  $\Bfree \pi$ and $\Bfree \Gamma_+$, $\Bfree \Gamma_-$ and $\Bfree H$.
  In the separating case, we have
  \begin{equation*}
    \Bfree \Gamma = \Bfree \Gamma_+ \cup_{\Bfree \pi\times \{-1\}} \Bfree\pi \times [-1,1]
    \cup_{\Bfree\pi\times\{1\}} \Bfree\Gamma_-
  \end{equation*}
  where we construct $\Bfree \Gamma_-, \Bfree\Gamma_-$ such that they contain copies of
  $\Bfree\pi$ which induce the inclusion maps on fundamental groups.
  More specifically, one can construct the classifying spaces in
  question as CW-complexes starting with $M$: first attach cells to $N$ to
  construct $B\pi$.
  Taking the product with $[-1,1]$ and glueing it into $M$
  along a trivialization of a tubular neighborhood of $N$ in $M$ produces
  $M\cup_{N\times [-1,1]} \Bfree \pi\times [-1,1]$.
  Now, attach further cells to $M_+\cup_{N\times\{1\}}\Bfree \pi$ to obtain
    $\Bfree \Gamma_+$ and  to $M_-\cup_{N\times\{-1\}}$, to obtain $\Bfree \Gamma_-$. Here,
    for convenience we slightly change notation and write $M=M_+\cup N\times
    [-1,1]\cup M_-$, glued along the two boundary components of $N\times
    [-1,1]$.

   The construction in the second case, where $N$ is not separating, is
   precisely the same. We then use the fact that the classifying space of an
   HNN-extension $H*_{\pi=\gamma^{-1}\pi\gamma}$ can be obtained from a
     classifying space of $H$ by glueing in $B\pi\times [-1,1]$, attaching the
     two ends according to the two inclusions of $\pi$ and
     $\gamma^{-1}\pi\gamma$ into $H$; and this can be modelled exactly as
     before, starting with $N\times [-1,1]\subset M$.

     In either case, we obtain a diagram of embeddings
     \[
      \begin{tikzcd}
        N \times [-1,1] \dar[hook]  \rar[hook]& \Bfree \pi \times [-1,1] \dar[hook] \\
        M \rar[hook] & \Bfree \Gamma
\end{tikzcd}     \]
     By definition, the collapse map $M\to \Sigma N$  maps every point inside $N \times (-1,1)$ of $N$ to the corresponding point in $N\times (-1,1)\subset \Sigma N$, and every point outside this tubular
    neighborhood to the (collapsed) base point in $\Sigma N$.
    By applying the same construction to $\Bfree \pi \times (-1,1) \subset \Bfree \Gamma$ we get a map $\Bfree \Gamma \to \Sigma \Bfree \pi$ and a commutative diagram
    \[
     \begin{tikzcd}
       M \dar  \rar[hook]& \Bfree \Gamma \dar \\
       \Sigma N \rar & \Sigma \Bfree \pi
\end{tikzcd}     \]
From this we obtain the commutative diagram
\begin{equation}\label{eq:codim1Top}
\begin{tikzcd}[ampersand replacement=\&]
  \Omega_*^{\spin}(M) \rar["c_\ast"] \dar["\tau_{M,N}"] \& \Omega_*^{\spin}(\Bfree \Gamma) \rar \dar["\tau_{\Gamma, \pi}"] \& \KO_\ast(\Bfree \Gamma) \dar["\tau_{\Gamma, \pi}"] \\
  \Omega_{\ast-1}^{\spin}(N) \rar["\bar{c}_\ast"] \& \Omega_{\ast-1}^{\spin}(\Bfree \pi) \rar \& \KO_{\ast-1}(\Bfree \pi).
\end{tikzcd}
\end{equation}
which consists precisely of the first three columns in \labelcref{eq:codim1Transfer}.

Finally, if the space $\Bfree \Gamma$ constructed above is compact (that is, we only had to add finitely many cells), then we are in \cref{setup:KKTransfer}.
The diagram \labelcref{eq:codim1Transfer} is completed by combining \labelcref{eq:codim1Top} with the diagram from \cref{thm:generalKKtransfer} (for $M = \Bfree \Gamma$ and $N = \Bfree \pi$).

In the general case, where $\Bfree \Gamma$ is non-compact, $\KO_\ast(\Bfree \Gamma)$ is identified with the colimit of $\KO_\ast(X)$, where $X$ runs over compact subsets of $\Bfree \Gamma$.
Then the colimit can be restricted to the directed set of those compact subsets $X$ such that $X \cap (\Bfree \pi \times [-1,1]) = Y \times [-1,1]$ for some compact $Y \subseteq \Bfree \pi$.
Since $\Gamma = \pi_1(M)$ is finitely presented and $\pi =\im(\pi_1(N)\to \pi_1(M))$ is finitely generated, we can arrange it so that we can further restrict the colimit to those pairs $(X, Y)$ with $\pi_1(X) = \Gamma$ and $\pi_1(Y) \twoheadrightarrow \pi$.
Then we may apply \cref{thm:generalKKtransfer} to each of these pairs $(X, Y)$ and the naturality statement of \cref{thm:generalKKtransfer} implies that this passes to the colimit and fits together with \labelcref{eq:codim1Top}, thereby completing \labelcref{eq:codim1Transfer}.
\end{proof}

\section{General transfer construction} \label{sec:generalTransfer}
In this section, we describe a general transfer construction which is used in the proof of the remaining results.
Suppose that we are in the situation of \cref{setup:geom} for some codimension \(k\).
Let $E$ be a multiplicative equivariant generalized homology theory (see, e.g., \cite{Lueck}).
Assume that \(\nu\) is $E$-oriented and let $\theta\in E^k(\disk\nu,\sphere\nu)$ be the corresponding Thom class.
Then the \emph{classical transfer map} (see \cite{RudyakThomSpectra}*{Chapter V.2}) is given by
\[\tau_\theta\colon\myE{\ast}(M)\xrightarrow{}\myE{\ast}(M,M\setminus \disk\nu^\circ)\iso\myE{\ast}(\disk\nu,\sphere\nu)\xrightarrow{\cap\theta}\myE{\ast-k}(N).\]
If \(M\) and \(N\) are \(E\)-oriented manifolds, then —fixing orientations on
\(M\) and \(N\)— we obtain an orientation on the normal bundle such that the corresponding transfer sends the fundamental class of $M$ to the fundamental class of $N$.

It is our goal to find conditions under which $\tau_\theta$ can be extended to the classifying spaces for free actions, or even the classifying spaces for proper actions, such that the following diagram commutes, where the horizontal arrows are induced by classifying maps:
\[\begin{tikzcd}
E_*(M)\arrow[r]\arrow[d,"\tau_\theta"] &
E_*(\Bfree{\Gamma})\arrow[r]\arrow[d,dashed] &
E^\Gamma_*(\Eub{\Gamma})\arrow[d,dashed]\\
E_{*-k}(N)\arrow[r] &
E_{*-k}(\Bfree{\pi})\arrow[r] &
E^\pi_{*-k}(\Eub{\pi})
\end{tikzcd}\]

To do so we will work in the equivariant setting.
Conceptually, our method uses \enquote{locally finite equivariant generalized homology theories}. But there does not seem to be an established set of axioms for this concept, and it would take us to far to introduce one. Instead, we give an ad-hoc definition for an \enquote{lf-restriction} operation (the \enquote{lf} stands for locally finite).

\begin{definition}\label{def:locally-finite}
Let $\Gamma$ be a countable discrete group, let $\pi$ be a subgroup, and let $Y$ be a proper $\Gamma$-CW-complex. A $\pi$-invariant subspace $K\subset Y$ is called \emph{$\Gamma$\nobreakdash-locally $\pi$\nobreakdash-precompact} if for every $\Gamma$-cell $X\subset Y$ the intersection $K\cap X$ is contained in only finitely many of the $\pi$-cells that together form $X$.
\end{definition}

\begin{definition}\label{def:lf-restrictions}
A generalized equivariant homology theory $E$ has \emph{lf-restrictions} if
for every inclusion of groups $\pi<\Gamma$, for every proper
$\Gamma$-CW-complex $Y$ and for every $\Gamma$-locally $\pi$-precompact subspace $K\subset Y$ such
that $(Y,Y\setminus K)$ is a pair of $\pi$-CW-complexes, there is a homomorphism
\[r_K\colon E_*^\Gamma(Y)\to E_*^\pi(Y,Y\setminus K).\]

These maps must be natural with respect to cellular equivariant maps of CW-pairs, and they must be compatible with the induction isomorphisms in the following sense:
Suppose that $\Gamma$ acts freely on $Y$, that $\overline{K}$ is a $\pi$-subcomplex of $Y$, and that $\Gamma\backslash \Gamma\cdot\overline{K}=\pi\backslash \overline{K}$.
The latter condition means $\gamma\cdot\overline{K}\cap \overline{K}\neq\emptyset$ is possible only for $\gamma\in\pi\subset\Gamma$.
Then the following diagram makes sense
\begin{equation*}
\resizebox{\linewidth}{!}{\begin{tikzcd}[ampersand replacement=\&]
E_*^\Gamma(Y) \arrow[r, "r_K"]\arrow{d}{\cong}[swap]{\mathrm{ind}} \&
E_*^\pi(Y,Y\setminus K)\arrow{r}{\cong}[swap]{\mathrm{exc}} \&
E_*^\pi(\overline{K},\overline{K}\cap(Y\setminus K))\arrow{d}{\cong}[swap]{\mathrm{ind}} \\
E_*(\Gamma\backslash Y) \arrow[r] \&
E_*(\Gamma\backslash Y,(\Gamma\backslash Y)\setminus(\Gamma\backslash
\Gamma\cdot K))\arrow{r}{\cong}[swap]{\mathrm{exc}} \& E_*(\pi\backslash\overline{K},\pi\backslash(\overline{K}\cap (Y\setminus K)))
\end{tikzcd}}
\end{equation*}
where $\mathrm{ind}$ denotes induction maps, $\mathrm{exc}$ denotes excision,
and the bottom left map is inclusion induced. We require that the diagram commutes.
\end{definition}

\begin{notation}
In \cref{setup:geom} we assume that $\disk\nu$, $M\setminus\disk\nu^\circ$ and $\sphere\nu$ are subcomplexes of $M$, and that the $\Gamma$-CW-structure on $\ucov{M}$ is the one lifted from $M$.
\end{notation}

We will show in \cref{sec:homologies-with-lf} that real K-homology, the most relevant example for $\myE{\ast}$, has lf-restrictions.

Assume now that $E_*$ is a generalized equivariant multiplicative homology theory with lf-restrictions, $\pi<\Gamma$ and $Y$ is a proper $\Gamma$-space. Given an element $\theta\in E^k_\pi(Y,\ccomp{Y}{K})$, with $K$ $\Gamma$-locally $\pi$-precompact, we define the equivariant transfer map
\[\tau^\mathrm{eq}_{\theta,K}\colon E_*^\Gamma(Y)\to E_{*-k}^\pi(Y),\qquad y\mapsto r_K(y)\cap \theta.\]
Because both the cap product and the lf-restriction are natural with respect to maps of pairs $(Y,\ccomp{Y}{K})$, the same is true for $\tau^\mathrm{eq}_{\theta,K}$. In particular, if $i\colon \ccomp{Y}{K'}\to\ccomp{Y}{K}$ is the inclusion of the complement of a larger $\Gamma$-locally $\pi$-precompact subspace $K'$, then $\tau^\mathrm{eq}_{i^*(\theta),K'}=\tau^\mathrm{eq}_{\theta,K}$. Thus, an element $\theta\in\colim_K E^*_\pi(Y,\ccomp{Y}{K})$ defines a map $\tau^\mathrm{eq}_\theta$, where the colimit runs over all $\pi$-subcomplexes $Y\setminus K$ with \(\Gamma\)-locally \(\pi\)-precompact complement.

If $f\colon X\to Y$ is a $\Gamma$-equivariant cellular map between proper $\Gamma$-CW-complexes, then for any $\Gamma$-locally $\pi$-precompact subset $K\subset Y$ one can form the subcomplex $L\subset X$ consisting of all cells whose image under $f$ does not intersect $K$.
It follows directly from \cref{def:locally-finite} that $X\setminus L$ is $\Gamma$-locally $\pi$-precompact. This means that the elements in $\colim_K E^*_\pi(Y,\ccomp{Y}{K})$ can be pulled back to $X$ via $f$.
Furthermore, the pullback only depends on the homotopy class of $f$.
Indeed, if $H\colon X\times[0,1]\to Y$ is a homotopy from $f$ to $f'$, we can
restrict ourselves to the smaller subcomplex $L_H\subset X$ of all cells whose
image under $f$ does not intersect $K$ during the whole homotopy. Then $X\setminus L_H$ is still $\Gamma$-locally $\pi$-precompact.

\medskip

We now establish a general-purpose result that will allow us to extend the classical transfer map in many cases.

\begin{theorem}\label{thm:transfer-extension}
Assume \cref{setup:geom}.
Let $\theta\in E^k(\disk\nu,\sphere\nu)\cong E^k_\pi(\ucov{\disk\nu},\ucov{\sphere\nu})\cong E^k_\pi(\ucov{M},\ccomp{\ucov{M}}{\ucov{\disk\nu}^\circ})$ and let $\tau_\theta\colon E_*(M)\to E_{*-k}(N)$ be the associated classical transfer map. Consider the following conditions:
\begin{enumerate}
\item\label{item:lift-to-Efree}
There exist $\Gamma$-locally $\pi$-precompact subspaces $K\supset\ucov{\disk\nu}$ and $K_\Efree\subset\Efree{\Gamma}$ such that the classifying map $\ucov{M}\to\Efree{\Gamma}$ is a map of $\pi$-CW-pairs $(\ucov{M},\ccomp{\ucov{M}}{K})\to(\Efree{\Gamma},\ccomp{\Efree{\Gamma}}{K_{\Efree}})$, and such that the restriction $i_{\ucov{\disk\nu},K}^*(\theta)\in E_\pi^k(\ucov{M},\ccomp{\ucov{M}}{K})$ can be lifted to $\theta_\Efree\in E^k_\pi(\Efree{\Gamma},\ccomp{\Efree{\Gamma}}{K_{\Efree}})$.
\item\label{item:lift-to-Eub}
There exist $\Gamma$-locally $\pi$-precompact subspaces $K_\Efree\subset K'_{\Efree}\subset\Efree{\Gamma}$ and $K_{\underline{\mathrm{E}}}\subset\Eub{\Gamma}$ such that the classifying map $\Efree{\Gamma}\to\Eub{\Gamma}$ is a map of $\pi$-CW-pairs $(\Efree{\Gamma},\ccomp{\Efree{\Gamma}}{K'_{\Efree}})\to(\Eub{\Gamma},\ccomp{\Eub{\Gamma}}{K_{\underline{\mathrm{E}}}})$, and such that the restriction $i_{K_{\Efree},K'_{\Efree}}^*(\theta_{\Efree})\in E_\Gamma^k(\Efree{\Gamma},\ccomp{\Efree{\Gamma}}{K'_{\Efree}})$ can be lifted to $\theta_{\underline{\mathrm{E}}}\in E^k_\pi(\Eub{\Gamma},\ccomp{\Eub{\Gamma}}{K_{\underline{\mathrm{E}}}})$.
\end{enumerate}
If condition \ref{item:lift-to-Efree} is satisfied, $\tau_\theta$ can be extended to $E_*(\Bfree{\Gamma})\to E_{*-k}(\Bfree{\pi})$. If, in addition, condition \ref{item:lift-to-Eub} is satisfied, $\tau_\theta$ can be extended further to $E_*^\Gamma(\Eub{\Gamma})\to E_{*-k}^\pi(\Eub{\pi})$.
\end{theorem}
\begin{proof}
Consider first the following diagram where $\theta'$ and $\theta''$ are the images of $\theta$ under the isomorphisms $E^k(\disk\nu,\sphere\nu)\cong E^k_\pi(\ucov{\disk\nu},\ucov{\sphere\nu})\cong E^k_\pi(\ucov{M},\ccomp{\ucov{M}}{\ucov{\disk\nu}^\circ})$:
\[\begin{tikzcd}
E_*(M)\arrow[d]\arrow[ddr, bend left, "\tau_\theta"]\arrow[ddddd, bend left, in=240, out=300, swap, "\mathrm{ind}^{-1}"] & \\
E_*(M,\ccomp{M}{\disk\nu}^\circ)\arrow{d}{\cong}[swap]{\mathrm{exc}} & \\
E_*(\disk\nu,\sphere\nu)\arrow{d}{\cong}[swap]{\mathrm{ind}^{-1}}\arrow[r,"\cap\theta"] &
E_{*-k}(\disk\nu)\arrow{d}{\cong}[swap]{\mathrm{ind}^{-1}}\\
E_*^\pi(\ucov{\disk\nu},\ucov{\sphere\nu})\arrow[r,"\cap\theta'"] &
E_{*-k}^\pi(\ucov{\disk\nu})\arrow{d}[swap]{i_*}\\
E_*^\pi(\ucov{M},\ccomp{\ucov{M}}{\ucov{\disk\nu}^\circ})\arrow{u}{\mathrm{exc}}[swap]{\cong}\arrow[r,"\cap\theta''"] &
E_{*-k}^\pi(\ucov{M})\\
E_*^\Gamma(\ucov{M})\arrow[u,"r_{\disk\nu}"]\arrow[ur, bend right, swap, "\tau^\mathrm{eq}_{\theta''}"] &
\end{tikzcd}\]
The left part of the diagram commutes by the compatibility of lf-restrictions with induction (setting $K=\ucov{\disk\nu}^\circ$). The top and bottom triangles commute by the definition of $\tau_\theta$ and $\tau^\mathrm{eq}_{\theta''}$, respectively. The top square commutes because the cap product is compatible with induction, the bottom square commutes because the cap product is natural.

Consider next the following two diagrams:
\[\begin{tikzcd}
E_*(M)\arrow[r,"\tau_\theta"]\arrow[d,swap,"\mathrm{ind}^{-1}"] &
E_{*-k}(N)\arrow[d,"i_*\circ\mathrm{ind}^{-1}"] &[1.5cm]
E_*(M)\arrow[r,"\tau_\theta"]\arrow[d,swap,"\mathrm{ind}^{-1}"] &
E_{*-k}(N)\arrow[d,"i_*\circ\mathrm{ind}^{-1}"]\\
E_*^\Gamma(\ucov{M})\arrow[r,"\tau^\mathrm{eq}_{\theta''}"]\arrow[d] &
E_{*-k}^\pi(\ucov{M})\arrow[d] &
E_*^\Gamma(\ucov{M})\arrow[r,"\tau^\mathrm{eq}_{\theta''}"]\arrow[d] &
E_{*-k}^\pi(\ucov{M})\arrow[d]\\
E_*^\Gamma(\Efree{\Gamma})\arrow{d}{\cong}[swap]{\mathrm{ind}^{-1}}\arrow[r,"\tau^\mathrm{eq}_{\theta_\Efree}"] &
E_{*-k}^\pi(\Efree{\Gamma})\arrow{d}{\mathrm{ind}^{-1}}[swap]{\cong} &
E_*^\Gamma(\Efree{\Gamma})\arrow{d}\arrow[r,"\tau^\mathrm{eq}_{\theta_\Efree}"] &
E_{*-k}^\pi(\Efree{\Gamma})\arrow{d}\\
E_*(\Bfree{\Gamma})\arrow[r,dashed] &
E_{*-k}(\Bfree{\pi}) &
E_*^{\Gamma}(\Eub{\Gamma})\arrow[r,"\tau^\mathrm{eq}_{\theta_{\underline{\mathrm{E}}}}"] &
E_{*-k}^\pi(\Eub{\pi})
\end{tikzcd}\]
In both diagrams the top square is given by the outer arrows of the previous diagram. Hence, it commutes. The middle squares in both diagrams, as well as the bottom square in the second diagram, commute by naturality of the equivariant transfer map. The dashed morphism in the left diagram is defined to make the bottom square commutative.

We note that $\Efree\Gamma$ and $\Eub\Gamma$ with the group action restricted to $\pi$ are models for $\Efree\pi$ and $\Eub\pi$, respectively, and the composition along the left and right columns in both diagrams are classifying maps. Hence the claims follow.
\end{proof}

\begin{remark}\label{rmk:transfer-compatible-with-codim-1}
Under certain conditions there is another canonical way to extend $\tau_\theta$ to $E_*(\Bfree{\Gamma})\to E_{*-k}(\Bfree{\pi})$:
Assume that we have models $\Bfree\pi\subset\Bfree{\Gamma}$ such that a neighborhood of $\Bfree{\pi}$ is homeomorphic to a $k$-disk bundle $\disk\nu_{\Bfree}$. Then any Thom class $\theta_\Bfree\in E^k(\disk\nu_\Bfree,\sphere\nu_\Bfree)$ gives rise to a transfer map $\tau_{\theta_\Bfree}\colon E_*(\Bfree{\Gamma})\to E_{*-k}(\Bfree{\pi})$.
Assume further that the classifying map $M\to\Bfree{\Gamma}$ restricts to a bundle map $\disk\nu\to\disk\nu_\Bfree$ and sends $M\setminus\disk\nu^\circ$ into $\Bfree{\Gamma}\setminus{\disk\nu_\Bfree}^\circ$, and that the Thom class $\theta$ is the pullback of $\theta_\Bfree$ under the classifying map. Then $\tau_{\theta_\Bfree}$ extends $\tau_\theta$.
This is precisely how the maps in \labelcref{eq:codim1Top} are obtained in the proof of \cref{theo:codim_1_long}.

The extension $\tau_{\theta_\Bfree}$ coincides with the one obtained from \cref{thm:transfer-extension}, with $\theta_\Efree$ the image of $\theta_\Bfree$ under the isomorphisms $E^k(\disk\nu_\Bfree,\sphere\nu_\Bfree)\cong
 E^k_\pi(\ucov{\disk\nu_{\Bfree}},
\ucov{\sphere\nu_{\Bfree}})\cong
 E^k_\pi(\ucov{\Bfree\Gamma},
\ccomp{\ucov{\Bfree\Gamma}}
{\ucov{\disk\nu_{\Bfree}}^\circ})$.
Indeed, the first diagram of the preceding proof, applied to the pair $\disk\nu_\Bfree\subset\Bfree\Gamma$, shows that $\mathrm{ind}^{-1}\circ\tau^\mathrm{eq}_{\theta_\Efree}\circ\mathrm{ind}=i_*\circ\tau_{\theta_\Bfree}$, where $i\colon\disk\nu_\Bfree\to\pi\backslash\Efree\Gamma$ is a homotopy equivalence.

In codimension $\geq 2$ it is generally not clear when the above conditions can be satisfied. This is discussed in \cite{NitscheThesis}*{5.3.9--11} for codimension $2$.
\end{remark}

\begin{remark}
If we assume that the lf-restrictions are transitive with respect to inclusions of subgroups $\pi'<\pi<\Gamma$, the equivariant transfer map will also be transitive in the sense that $\tau^\mathrm{eq}_{\theta'}\circ\tau^\mathrm{eq}_\theta=\tau^\mathrm{eq}_{\theta'\cup\theta}$.
We will not make use of this in the following.
\end{remark}

\subsection{Homology theories with lf-restrictions}\label{sec:homologies-with-lf}

\begin{lemma}\label{lem:sing-homology-has-lf}
Bredon homology with constant coefficients and trivial group action on the coefficients has lf-restrictions.
\end{lemma}
\begin{proof}
Let $\pi,\Gamma,Y,K$ be as in \cref{def:lf-restrictions}. An element $x\in \HZ^\Gamma_*(Y)$ can be represented by a cycle $\sum_i \lambda_i c^\Gamma_i$, where the $\lambda_i$ lie in the coefficient module, $c^\Gamma_i$ are $\Gamma$-cells of $Y$ and the sum is finite.
We denote by $C^\pi_i$ the formal sum of $\pi$-cells of $Y$ that have non-empty intersection with $c^\Gamma_i\cap K$. This sum is finite because $K$ is $\Gamma$-locally $\pi$-precompact.
Let now $r_K(x)=\left[\sum_i \lambda_i C^\pi_i\right]\in \HZ^\pi_*(Y,\ccomp{Y}{K})$. One can easily check that the map $r_K$, thus defined, satisfies the conditions of \cref{def:lf-restrictions}.
\end{proof}

To show that real K-homology has lf-restrictions we use its geometric description.
For a detailed definition and a proof of equivalence to the analytic description see Baum--Higson--Schick, \cite{BaumHigsonSchick2} (They consider the complex case, but the real case is analogous, see their Remark 4.1).

\begin{definition}\label{def:geometric-K-homology}
	Let $(X,A)$ be a pair of topological spaces and $\Gamma\curvearrowright(X,A)$ a proper action of a discrete group.
	Elements of $\KO_n^\Gamma(X,A)$ are represented by quadruples $(M,s,E,f)$, where $M$ is a $\Gamma$-compact, proper $\Gamma$-manifold of dimension $n$ mod $8$, $s$ is a $\Gamma$-$\spin$-structure on $M$, $E\to M$ a real $\Gamma$-vector bundle and $f\colon M\to X$ a continuous $\Gamma$-equivariant map such that $f(\partial M)\subset A$.
	The equivalence relation is generated by:
	\begin{enumerate}[itemsep=0pt]
		\item
		direct sum of vector bundles equals disjoint union,
		\item
		equivariant bordism,
		\item
		equivariant vector bundle modification.
	\end{enumerate}
\end{definition}

\begin{lemma}\label{lem:K-homology-has-lf}
	Real K-homology has lf-restrictions.
\end{lemma}
\begin{proof}
	Let $\pi,\Gamma,Y,K$ be as in \cref{def:lf-restrictions} and let
	$[M,s,E,f]\in\KO_*^\Gamma(Y)$.

	To define the lf-restriction map we have to find any $\pi$-invariant $\pi$-compact submanifold with boundary $M'\subset M$ of codimension $0$, such that $f(M\setminus M')\subset \ccomp{Y}{K}$. The restricted K-homology cycle $(M',s_{|M'},E_{|M'},f_{M'})$ with an action of only $\pi$ then represents an element in $\KO_*^\pi(Y,\ccomp{Y}{K})$ that does not depend on the concrete choice of $M'$. The construction is compatible with taking disjoint unions, with vector bundle addition and with vector bundle modification. Because the method given below to find $M'$ can also be applied to bordisms, the construction is compatible with the bordism relation and therefore defines a group homomorphism from $\KO_*^\Gamma(Y)$ to $\KO_*^\pi(Y,\ccomp{Y}{K})$.


	To construct the submanifold $M'\subset M$, we first check that $L:=\overline{f^{-1}(K)}$ is $\pi$-compact. Indeed, if $F\subset M$ is the closure of a $\Gamma$-fundamental domain, then $F$ is compact. Hence, $\overline{\gamma}\cdot F\cap f^{-1}(K)\neq\emptyset$ only for finitely many cosets $\overline{\gamma}\in\overline{\Gamma}_F\subset\Gamma/\pi$, and it follows that $f^{-1}(K)\subset\bigcup_{\overline{\gamma}\in\overline{\Gamma}_F}\overline{\gamma}\cdot F$ is $\pi$-precompact.
	We choose a $\Gamma$-invariant Riemannian metric on $M$ and consider the function $d\colon M\to\R$ that assigns to $m\in M$ the distance from $m$ to $L$. This function is $\pi$-equivariant. The induced function $\overline{d}\colon \pi\backslash M\to\R$ is the distance function (in the metric induced from $M$) to the compact set $\pi\backslash L$. Hence $\overline{d}$ is proper.

	Next, we need a $\pi$-invariant smooth approximation of $d$. Because the action $\Gamma\curvearrowright M$ is proper, every point $m\in M$ has a $\Gamma$-invariant neighborhood of the form $\bigsqcup_{[\gamma]\in\Gamma/\Gamma_m} W_{m,[\gamma]}$, where $\Gamma_m$ is the stabilizer of $m$. Because $M$ is $\Gamma$-compact, it is covered by a finite set of such neighborhoods. Because all stabilizer groups are finite, one can construct a $\Gamma$-invariant smooth partition of unity subordinate to the covering. Using the partition of unity we can now smoothen $d$ separately on each neighborhood. This works by choosing a smoothing on one $W_{m,\overline{\gamma}}$ for each coset $\overline{\gamma}\in\pi\backslash\Gamma/\Gamma_m$, averaging these over the action of $\pi_m$ and extending them $\pi$-equivariantly to the whole $\Gamma$-neighborhood.

	Finally, pick any regular value $r>0$ of $d$ such that $f(d^{-1}([r,\infty)))\subset \ccomp{Y}{K}$. The preimage $M'=d^{-1}((-\infty,r])\subset M$ is a $\pi$-invariant submanifold with boundary. Because $\overline{d}$ is still a proper function after smooth approximation, $M'$ is $\pi$-compact.

	The lf-restriction map, thus constructed, is clearly natural with respect to maps of CW-pairs. To see that it is compatible with induction, assume now that
	$\Gamma$ acts freely on $Y$ and that $\overline{K}$ is a subcomplex with $\Gamma\backslash\Gamma\cdot\overline{K}=\pi\backslash\overline{K}$, as in \cref{def:lf-restrictions}.
	We can construct a $\pi$-invariant neighborhood $U\supset\overline{K}$ such that $\Gamma\backslash\Gamma\cdot U=\pi\backslash U$ and such that there is a $\pi$-equivariant deformation retraction $H_t$ from $U$ to $\overline{K}$.
	Then $M\setminus f^{-1}(U)$ has a positive distance to $L$, and so we may arrange that $f(M')\subset U$.
	The image $\mathrm{ind}\circ\mathrm{exc}\circ r_K([M,s,E,f])$ is
        represented by the quotient of $(M',s_{|M'},E_{|M'},H_1\circ f_{M'})$
        by $\pi$. But the same representative can be obtained by first taking
        the quotient by $\Gamma$, then cutting $\Gamma\backslash M$ off to
        $\Gamma\backslash \Gamma\cdot M'=\pi\backslash M'$ and applying the deformation retraction induced by $H_t$ on the quotient.
	Therefore, one sees on the level of cycles that the diagram of \cref{def:lf-restrictions} is commutative.
\end{proof}

\begin{remark}\label{rem:cobordism-has-lf}
The above proof also works for complex K-homology and for all bordism theories.
Moreover, Jakob showed in \cite{Jakob} that all multiplicative non-equivariant generalized homology theories have a bordism-type description similar to that of K-homology. Thus, it seems likely that all these homology theories admit an equivariant extension that has lf-restrictions.
\end{remark}

\section{Applying the general transfer construction}
\label{sec:applications}

We apply \cref{thm:transfer-extension} to specific scenarios where we put certain restrictions on the homology theory $E$, on the codimension of the subspace $N\subset M$ or on the normal bundle of $N$.

In each case we will describe the cohomology class $\theta$ in terms of a pullback of a fixed cohomology class $x\in E^*(\mathrm{K}(\Z,n))$ along some classifying map. The extensions of $\theta$ are then obtained as the pullback of $x$ under extensions of the classifying map.
This approach has the additional benefit that it is natural with respect to transformations of homology theories.

\subsection{Transfer in singular homology} \label{subsec:application-singular-homology}

The following is a restatement of \cref{theo:codim_homol}.

\begin{theorem}\label{thm:application-singular-homology}
Suppose that we are in \cref{setup:geom} with $N \subset M$ an embedding of codimension $k$ with normal bundle $\nu$.
Let $E_*=\HZ_*(\blank;A)$ be singular homology with coefficients in a commutative ring $A$ and let $\nu$ be oriented with Thom class $\theta\in \HZ^k(\disk\nu,\sphere\nu; A)$.

  In addition, we assume that
  \begin{enumerate}
  \item if $k>1$, then $\pi_k(N)\to \pi_k(M)$ is surjective; \label{item:HomotopySurj}
  \item if \(k>2\), then $\pi_j(M)=0$ for $j=2,\dots,k-1$; \label{item:HomotopyZero}
  \item there exists a subset \(S \subseteq \pi_k(N)\) which generates \(\pi_k(M)\) such that the composition of the Hurewicz homomorphism $\pi_k(N)\to \HZ_k(N;\Z)$ with the Euler class of $\nu$ vanishes on \(S\).
In particular, this is satisfied if the normal bundle is trivial or \(\pi_k(M)=0\) or \(\dim N < \frac{\dim M}{2}\). \label{item:EulerClassCondition}
  \end{enumerate}

Then the classical transfer map can be extended to the classifying space for free actions.
\end{theorem}
\begin{proof}
We have to extend $\theta$ from $\HZ^k_\pi(\ucov{\disk\nu},\ucov{\sphere\nu})$ to $\HZ^k_\pi(\Efree{\Gamma},\ccomp{\Efree{\Gamma}}{K})$ for a suitable $\Gamma$-locally $\pi$-precompact subspace $K \subseteq \Efree \Gamma$. Then the result follows from \cref{thm:transfer-extension}.

We start with $K=\ucov{\disk\nu}$ and build $\Efree{\Gamma}$ out of $\ucov{M}$ by inductively attaching free $\Gamma$-cells of dimension $>k$.
Each new $\Gamma$-cell decomposes into $\pi$-cells. Because $K$ is $\Gamma$-locally $\pi$-precompact and the attaching maps are $\Gamma$-equivariant, all but finitely many of the $\pi$-cells do not attach to $K$.
The interior of the other $\pi$-cells has to be added to $K$, but since these are only finitely many, $K$ stays $\Gamma$-locally $\pi$-precompact.

To complete the inductive step, $\theta$ has to be extended to the added cells.
If the new $\pi$-cells have dimension $>k+1$, this extension is possible because singular homology satisfies the dimension axiom. In dimension $k+1$ we may assume that the attaching map represents an element of $S\subset\pi_k(N)$. The extension is then possible by the condition on the Euler class.
\end{proof}

\begin{remark}
For singular homology our extension of the transfer map coincides with the map
constructed by Engel in the setting of rough homology in \cite{Engel}:
Without recapitulating all the definitions, we note that, conceptually, Engel's map is given by cap product with a certain cohomology class that arises from the Thom class (see \cite{Engel}*{Proof of Theorem 2.16}), and the same is true for our construction.

To verify that the two maps are the same one has to retrace the proof of \cite{Engel}*{Theorem 3.10}. Without loss of generality, $\Efree\pi_1(M)$ is also used as a model for $\Efree\pi_1(N)$, and $\Bfree\pi_1(M)$ contains $M$ as a subcomplex.
To evaluate Engel's map on a homology class $x\in\HZ_*(\Bfree\pi_1(M))$ we first represent $x$ by a cycle, such that the vertices of all singular simplices map to a fixed base point in $M$. Because $M$ is highly connected, we can further assume that the $q$-faces of every simplex are mapped to $M$.
Now we trace through all the isomorphisms involved in Engel's construction.
Whenever $\pi_1(M)$ is embedded into $\Efree\pi_1(M)$ as the orbit of a point, we let this point be a fixed lift of the base point. Whenever the simplicial operator $\Delta$ of \cite{Engel}*{Proof of Proposition 2.8} is applied to the vertex set $\overline{x}$ of a singular simplex, we may assume that $\Delta(\overline{x})$ is the original simplex. In the end, the image of $x$ agrees with the image under our own extended transfer map.
\end{remark}

Although the hypotheses of \cref{theo:codim_homol} appear elaborate, the following examples demonstrate why they are necessary.

\begin{example}
  Let \(N\) be an oriented aspherical manifold and \(k > 1\).
  Then the codimension \(k\) submanifold inclusion \(N \cong N \times \ast \subset N \times \sphere^k \eqqcolon M\) with \(\pi = \pi_1 N = \pi_1 M = \Gamma\) satisfies all conditions of \cref{theo:codim_homol} except \labelcref{item:HomotopySurj}.
  However, the desired commutative diagram cannot exist
  because the fundamental class of \(N\) does not vanish in the homology of \(\Bfree \pi = N\) but the fundamental class of \(N \times \sphere^k\) does vanish in \(\HZ_\ast(\Bfree \pi)\).
\end{example}

\begin{example} \label{ex:HomotopyZero}
  To see why \labelcref{item:HomotopyZero} is necessary, consider the standard embedding \(N \coloneqq \sphere^2 = \CP^1 \subset \CP^3 \eqqcolon M\).
  Then \(N\) has codimension four in \(M\), \(\pi_1 N = \pi_1 M = 1\) and the normal bundle is orientable because \(\CP^1\) and \(\CP^3\) are.
  Since \(\pi_4( \CP^3) = 0\), condition \labelcref{item:HomotopySurj} is satisfied.
  So is \labelcref{item:EulerClassCondition} due to \(\HZ^4(\CP^1) = 0\).
  However, \labelcref{item:HomotopyZero} fails because \(\pi_2(\CP^3) \cong \Z \neq 0\).
    Since \(\CP^2\) and \(\CP^1\) have intersection number \(1\) in \(\CP^3\) and \(\HZ_4(\CP^3)\) is generated by \(\CP^2\), in degree \(4\) the transfer map of the inclusion \(\CP^1 \subset \CP^3\) yields an isomorphism \(\HZ_4(\CP^3) \xrightarrow{\cong} \HZ_0(\CP^1) \cong \HZ_0(\ast) \cong \Z\).
    But \(\HZ_4\) of the trivial group is zero.
    Hence the diagram
  \[ \begin{tikzcd}
      \HZ_4(\CP^3) \rar \dar["\cong"] & \HZ_4(\ast) = 0 \dar[dashed]\\
      \HZ_{0}(\CP^1) \rar["\cong"] & \HZ_{0}(\ast) \cong \Z
    \end{tikzcd} \]
    cannot be completed with a dashed arrow to make it commutative.
\end{example}

A priori, a plausible alternative to condition \labelcref{item:HomotopyZero} might be to require that the induced maps \(\pi_j(N) \to \pi_j(M)\) are surjective (or isomorphisms) for \(j = {2, \dotsc, k-1}\).
For instance, this is suggested by the case of fiber bundles over aspherical manifolds which have been studied in \cite{ZeidlerSubmf}*{Theorem~1.5}.
But \cref{ex:HomotopyZero} would satisfy such a hypothesis and hence this is not appropriate for our purposes.

\begin{example} \label{ex:EulerClassCondition}
  A variation of \cref{ex:HomotopyZero} shows that condition \labelcref{item:EulerClassCondition} cannot be dropped either.
  Indeed, consider the standard embedding \(N \coloneqq \sphere^2 = \CP^1 \subset \CP^2 \eqqcolon M\).
  Then \(N\) has codimension two in \(M\),  \(\pi_1 N = \pi_1 M = 1\) and the Hurewicz theorem implies that the inclusion \(N \hookrightarrow M\) induces an isomorphism on \(\pi_2\) because it does so on \(\HZ_2\).
  Now this example satisfies all hypotheses of \cref{theo:codim_homol} except \labelcref{item:EulerClassCondition}.
  Since \(\CP^1\) has self-intersection number \(1\) in \(\CP^2\) and \(\HZ_2(\CP^2)\) is generated by \(\CP^1\), in degree \(2\) the transfer map of the inclusion \(\CP^1 \subset \CP^2\) yields again an isomorphism \(\HZ_2(\CP^2) \xrightarrow{\cong} \HZ_0(\CP^1) \cong \HZ_0(\ast) \cong \Z\).
  Since \(\HZ_2(\ast) = 0\), this again obstructs the existence of the desired commutative diagram.
\end{example}

\begin{corollary}\label{cor:application-singular-homology}
  Assume that we are in the situation of \cref{thm:application-singular-homology} \parensup{with \(A=\Z\)}, that the assembly map \(\mu \otimes \Q \colon \KO_\ast(\Bfree \Gamma) \otimes \Q \to \KO_\ast(\Cstar \Gamma) \otimes \Q\) is injective, and that $M$ is spin. Moreover, assume that either
\begin{enumerate}\setlength\parskip{0pt}\setlength\itemsep{0pt}
\item\label{item:Essentialness} the push forward of the homological fundamental class of \(N\) does not vanish in \(\HZ_\ast(\Bfree \pi; \Q)\), or
\item\label{item:HigherAHat} the normal bundle of $N\subset M$ is trivial and the push forward of the \textup{KO}-fundamental class of \(N\) \parensup{with the induced spin structure} in \(\KO_\ast(\Bfree \pi) \otimes \Q\) does not vanish.
\end{enumerate}
Then $\alpha(M)\neq 0$ and hence \(M\) does not admit a metric of positive scalar curvature.
\end{corollary}
\begin{proof}
On complex \(\K\)-homology the Chern character is a transformation
\[\ch \colon \K_n(X) \to \bigoplus_{l \in \Z} \HZ_{n+2l}(X; \Q)\]
which becomes an isomorphism after taking the tensor product with \(\Q\).
Precomposing the Chern character with the complexification transformation on real \(\K\)-homology yields the Pontryagin character
\[
  \ph \colon \KO_n(X) \to \bigoplus_{l \in \Z} \HZ_{n+4l}(X; \Q)
\]
which again becomes an isomorphism after taking the tensor product with \(\Q\).

The composition $\Omega^{\spin}_n(X)\to \KO_n(X)\xrightarrow{\ph}
\bigoplus_{l \in \Z}\HZ_{n+4l}(X; \rationals)$ agrees in degree \(n\) with the
composition $\Omega_n^{\spin}(X)\to
\Omega_n^{\mathrm{SO}}(X)\to \HZ_n(X;\integers)\to \HZ_n(X;\rationals)$ which assigns to
  an oriented manifold its homological fundamental class.
  However, in general they are not equal because the former contains information determined by the \(\hat{\mathrm{A}}\)-class of the tangent bundle in the lower degree summands.
  Similarly, if \(N \subset M\) is a codimension \(k\) submanifold embedding with \(\KO\)-oriented normal bundle,
  then the diagram
  \begin{equation}
    \begin{tikzcd}
      \KO_n(M) \rar["\ph"] \dar["\tau_{M,N}"] & \bigoplus_{l \in \Z} \HZ_{n+4l}(M; \Q) \dar["\tau_{M,N}"] \\
      \KO_{n-k}(N) \rar["\ph"] & \bigoplus_{l \in \Z} \HZ_{n-k+4l}(N; \Q)
      \end{tikzcd}
      \label{eq:ChernCharacterDefect}
  \end{equation}
  does not commute in general because of the Chern character defect.
  But if \(M\) is spin, then the classes \(\tau_{M,N} \ph ([M]_\KO)\) and \(\ph \tau_{M,N}  ([M]_\KO)\) agree in the top degree because they are both equal to the homological fundamental class of \(N\).
  This observation together with \cref{thm:application-singular-homology} proves \labelcref{item:Essentialness}.

  Finally, in the special case of a trivial normal bundle, the diagram \labelcref{eq:ChernCharacterDefect} commutes.
  Together with the fact that the Pontryagin character is a rational isomorphism and \cref{thm:application-singular-homology}, this proves \labelcref{item:HigherAHat}.
\end{proof}

\subsection{Transfer in a complex oriented homology theory}
  In this section, we apply our construction in codimension two and three to complex oriented cohomology theories.
  A classic textbook treatment of this concept is \cite{Adams}*{Part II}.
  Alternatively, see for instance~\cite{LurieChromatic}*{Lectures 4–6}.

    \begin{definition}
    Let $E$ be a generalized multiplicative cohomology theory. A \emph{complex
      orientation} for $E$ is a class $x\in E^2(\CP^\infty,\pt)$ which
    restricts to the twice-suspended unit in $E^2(\CP^1,\pt)=E^2(\sphere^2,\pt)$.
    $E$ is called \emph{complex orientable} if such an element exists.
  \end{definition}

  \begin{remark}[{\cite{Adams}*{Part II, Lemma~4.6}, \cite{LurieChromatic}*{Lecture 6, Theorem 8}}]
  Complex bordism is the universal example of a complex oriented cohomology theory.
  Indeed, complex orientations on a multiplicative cohomology theory \(E\) are in bijection with multiplicative transformations of cohomology theories \(f \colon \mathrm{MU} \to E\).
  There is a universal complex orientation \(x^{\mathrm{MU}} \in \mathrm{MU}^2(\CP^\infty, \pt)\) such that this bijection is implemented by taking \(f\) to \(f(x^{\mathrm{MU}})\).
\end{remark}

  \begin{example}
    The transformations $\mathrm{MU} \to \mathrm{MSpin}^{\mathrm{c}} \to \mathrm{MSO}$ show that the \(\mathrm{spin}^{\mathrm{c}}\)- and oriented bordism theories are complex oriented.
    The Atiyah–Bott–Shapiro orientation $\mathrm{MSpin}^{\mathrm{c}} \to \K$ makes complex K-theory complex oriented.

    On the other hand, $\spin$-bordism and real K-theory are both \emph{not} complex oriented.
    But they become complex oriented after inverting \(2\).
    For spin bordism, this is the case because the map \(\mathrm{MSpin} \to \mathrm{MSO}\) is an equivalence after inverting \(2\) and \(\mathrm{MSO}\) is complex oriented.
    Then it also follows for \(\KO\) via the Atiyah–Bott–Shapiro orientation $\mathrm{MSpin} \to \KO$.
  \end{example}

An ordinary orientation on a two-dimensional vector bundle $\nu$ over $N$ determines a complex structure on the bundle. Consequently, there is a homotopy class of classifying maps $f\colon (\disk\nu,\sphere\nu)\to(\disk\gamma,\sphere\gamma)\simeq (\CP^\infty,\sphere^\infty) $, where $\gamma$ is the universal complex line bundle over $\CP^\infty$.
Using that \(\sphere^\infty\) is contractible, we can identify \(f\) with a map \(f\colon (\disk\nu,\sphere\nu)\to(\CP^\infty,\pt)\).
The condition that $x\in E^2(\CP^\infty,\pt)$ is a complex orientation implies that $f^*(x)\in E^2(\disk\nu,\sphere\nu)$ is a Thom class.
Because $\CP^\infty=K(\Z,2)$, the homotopy class $[f]$ corresponds to an element in $\HZ^2(\disk\nu,\sphere\nu)$. This is the Thom class in ordinary cohomology for the given choice of ordinary orientation.
The restriction $\HZ^2(\disk\nu,\sphere\nu)\to \HZ^2(\disk\nu)\to \HZ^2(N)$ maps $[f]$ to the (ordinary) Chern class of $\nu$.

\begin{theorem}\label{thm:application-oriented-codim-2}
  Let $E$ be a multiplicative equivariant homology theory furnished with a complex orientation $x \in E^2(\CP^\infty,\pt)$ and lf-restrictions.

Suppose that we are in \cref{setup:geom} with $N \subset M$ an embedding of codimension $k = 2$.
Let the normal bundle \(\nu\) be oriented with ordinary Thom class $[f]\in \HZ^2(\disk\nu,\sphere\nu)$, viewed as a homotopy class of maps $f\colon (\disk\nu,\sphere\nu)\to(\CP^\infty,\pt)$.
Define the corresponding \(E\)-Thom class by $\theta \coloneqq f^*(x) \in E^2(\disk \nu, \sphere \nu)$.

Assume that the induced map $\pi_2(N)\to\pi_2(M)$ is surjective and that there exists a subset \(S \subseteq \pi_2(N)\) which generates \(\pi_2(M)\) such that the composition of the Hurewicz homomorphism with the Chern class of $\nu$ vanishes on \(S\).

Then the classical transfer map in \(E\)-homology defined by \(\theta\) can be extended to the classifying space for free actions. The extension is natural with respect to transformations of homology theories \parensup{that preserve all relevant structure, including the complex orientation}.
\end{theorem}
\begin{proof}
The same proof as for \cref{thm:application-singular-homology} with $k=2$ shows that we can extend $[f]\in \HZ^2(\disk\nu,\sphere\nu)\cong \HZ^2_\pi(\ucov{M},\ccomp{\ucov{M}}{\ucov{\disk\nu}})$ to $[f_\Efree]\in \HZ^2_\pi(\Efree{\Gamma},\ccomp{\Efree{\Gamma}}{K})$ for some \(\Gamma\)-locally \(\pi\)-precompact subspace \(K \subset \Efree \Gamma\).
Then we view $[f_\Efree]$ as a homotopy class of $\pi$-invariant maps of pairs $f_\Efree\colon (\Efree{\Gamma},\ccomp{\Efree{\Gamma}}{K})\to (\CP^\infty,\pt)$, and we set $\theta_\Efree={f_\Efree}^*(x)$. This is the extension of $\theta$ required to apply \cref{thm:transfer-extension}.
\end{proof}

Under more restrictive hypotheses, we also obtain a codimension three transfer for complex orientable cohomology theories.
\begin{theorem}\label{thm:application-oriented-codim-3}
  Let $E$ be a complex orientable multiplicative equivariant homology theory with lf-restrictions.

Suppose that we are in \cref{setup:geom} with $N \subset M$ an embedding of codimension $k = 3$.
Let the normal bundle \(\nu\) be trivialized.
Using the chosen trivialization, let \(\theta \in E^3(\disk \nu, \sphere \nu) \cong E^3(N \times \disk^3, N \times \sphere^2)\) be the pullback of the three times suspended unit in \(E^3(\disk^3, \sphere^2)\).

Assume that the induced map \(\pi_1(N) \to \pi_1(M)\) is injective, $\pi_2(N)= 0$, $\pi_2(M) = 0$, and that the induced map $\pi_3(N)\to\pi_3(M)$ is surjective.

Then the classical transfer map in \(E\)-homology defined by \(\theta\) can be extended to the classifying space for free actions. The extension is natural with respect to transformations of homology theories \parensup{that preserve all relevant structure, including the complex orientation}.
\end{theorem}
\begin{proof}
Let \(f \colon \sphere\nu \to \CP^\infty\) the map obtained as the composition of the projection \(\sphere \nu \to \sphere^2\) defined by the trivialization and the inclusion \(\sphere^2 = \CP^1 \subset \CP^\infty\).
Since \(E\) is complex oriented, the twice suspended unit in \(E^2(\sphere^2)\) is the restriction of a class \(x \in E^2(\CP^\infty)\).
This implies that \(\theta = \delta f^\ast(x)\), where \(\delta \colon E^2(\sphere \nu) \to E^3(\disk \nu, \sphere \nu)\) is the coboundary map associated to the pair \((\disk \nu, \sphere \nu)\).

Now consider the lifted embedding $\disk\nu\subset\bar{M} = \pi \backslash \ucov{M}$.
From the long exact sequence of homotopy groups and our assumptions on $\pi_1,\pi_2$ and $\pi_3$ it follows that $\pi_i(\bar{M},\disk\nu)=0$ for $i\leq 3$.
The relative Hurewicz Theorem then implies that $\HZ_i(\bar{M},\disk\nu)=0$ for \(i \leq 3\).
Applying the universal coefficient theorem and excision yields $\HZ^i(\ccomp{\bar{M}}{\disk\nu^\circ},\sphere\nu)\cong \HZ^i(\bar{M},\disk\nu)=0$ for $i\leq 3$.
From the long exact sequence of cohomology it follows that the inclusion $\sphere\nu\subset\ccomp{\bar{M}}{\disk\nu^\circ}$ induces an isomorphism $\HZ^2(\ccomp{\bar{M}}{\disk \nu^\circ}) \xrightarrow{\cong} \HZ^2(\sphere \nu)$.
This proves that \(f \colon \sphere \nu \to \CP^\infty \) can be extended to a map \(f_{\bar{M}} \colon \ccomp{\bar{M}}{\disk\nu^\circ} \to \CP^\infty\).
Lifting this to \(\ucov{M}\) yields a \(\pi\)-invariant map \(f_{\ucov{M}} \colon \ccomp{\ucov{M}}{\ucov{\disk \nu}^\circ} \to \CP^\infty\) which represents a class \([f_{\ucov{M}}] \in \HZ_\pi^2(\ccomp{\ucov{M}}{\ucov{\disk\nu}^\circ}) \).

Next, we build $\Efree{\Gamma}$ by attaching free $\Gamma$-cells of dimension $\geq 4$ to $\ucov{M}$.
During this process, we also build a suitable \(\Gamma\)-locally \(\pi\)-precompact subset \(K \subseteq \Efree \Gamma\) and extend $[f_\ucov{M}]\in \HZ^2_\pi(\ccomp{\ucov{M}}{\ucov{\disk\nu}^\circ})$ to $[f_\Efree] \in \HZ^2_\pi(\ccomp{\Efree{\Gamma}}{K})$.
 This works exactly as in \cref{thm:application-singular-homology} —note that there is no obstruction to extending a degree \(2\) cohomology class along the added cells because they have dimension at least \(4\).

Finally, we apply \cref{thm:transfer-extension}, where the required extension of $\theta$ is given by $\theta_\Efree \coloneqq \delta f_\Efree^*(x) \in E^3_\pi(\Efree \Gamma, \ccomp{\Efree \Gamma}{K})$ using \(\delta \colon E^2_\pi(\ccomp{\Efree \Gamma}{K}) \to E^3_\pi(\Efree \Gamma, \ccomp{\Efree \Gamma}{K}) \) associated to the pair \((\Efree \Gamma, \ccomp{\Efree \Gamma}{K})\).
\end{proof}

\subsection{Extending the equivariant transfer map to \texorpdfstring{$\Eub\Gamma$}{E bar Gamma}}\label{subsec:extension-to-eub}
Finally, we study the question when the equivariant transfer map can be further extended to the classifying space for proper actions.
In all previous applications we have reduced the construction of the transfer map to the task of extending a class in ordinary cohomology. Now we want to extend this class to $\Eub\Gamma$. As usual, this is done by a spectral sequence argument. But we have to be careful with the choice of our models for the classifying spaces.

\medskip

By the discussion preceding \cref{thm:transfer-extension} the equivariant transfer map does not rely on a particular choice of model for $\Efree\Gamma$ or $\Eub\Gamma$.
For $\Efree\Gamma$ we take Milnor's infinite join construction (i.e.\ the \enquote{fat model}) with the usual $\Gamma$-CW-structure, where $n$-simplices correspond to $n+1$-tuples in $\Gamma$. For $\Eub\Gamma$ we take the \enquote{geometric model}, whose $n$-simplices correspond to finite subsets of $\Gamma$ of cardinality $n+1$ (see Mislin's appendix to \cite{Valette}). After passing to the barycentric subdivision this becomes a $\Gamma$-CW-complex. The canonical map $q\colon\Efree\Gamma\to\Eub\Gamma$ is obtained by extending the obvious map on the $0$-skeleton affinely to all higher cells. Crucially, it is proper when restricted to the skeleta $\Efree\Gamma^{(k)}$.

\begin{lemma}\label{lem:adapt-K-to-model}
Let $\Efree\Gamma$ and $\Eub\Gamma$ be represented by the models above, let $\Efree\Gamma\setminus K_\Efree$ be a subcomplex with $\Gamma$-locally $\pi$-precompact complement, and let $n\in\N$.

Then there exists a $\pi$-subcomplex $L\subset\Eub\Gamma$ with $\Gamma$-locally $\pi$-precompact complement, such that $q^{-1}(L)\subset\Efree\Gamma$ is a $\pi$-subcomplex with $\Gamma$-locally $\pi$-precompact complement, and such that the restriction maps
\begin{align*}
\HZ_\pi^{n}((\Efree\Gamma\setminus K_\Efree)\cup q^{-1}(L);\Z) &
\xrightarrow{\hspace{0.5cm}}\HZ_\pi^{n}(\Efree\Gamma\setminus K_\Efree;\Z),\\
\HZ_\pi^{n+1}(\Efree\Gamma,(\Efree\Gamma\setminus K_\Efree)\cup q^{-1}(L);\Z) &
\xrightarrow{\hspace{0.5cm}}\HZ_\pi^{n+1}(\Efree\Gamma,\Efree\Gamma\setminus K_\Efree;\Z)
\end{align*}
are surjective.
\end{lemma}
\begin{proof}
First, we enlarge $\Efree\Gamma\setminus K_\Efree$ to a bigger subcomplex $\Efree\Gamma\setminus K'_{\Efree}$ such that every $\pi$-cell of dimension $>n+1$ whose boundary completely lies in $\Efree\Gamma\setminus K'_{\Efree}$ does itself belong to $\Efree\Gamma\setminus K'_{\Efree}$. We do this by induction over the skeleta of $\Efree\Gamma$, in each step adding all necessary $\pi$-cells to $\Efree\Gamma\setminus K'_{\Efree}$. If a cohomology class in $\HZ_\pi^n(\Efree\Gamma,\Efree\Gamma\setminus K_\Efree;\Z)$ or $\HZ_\pi^{n+1}(\Efree\Gamma\setminus K_\Efree;\Z)$ is given, it can, in each induction step, be simultaneously extended to all added $\pi$-cells. In fact, the extension is unique in each step, and by the $\operatorname{lim}^1$-short exact sequence and the Mittag-Leffler condition it is also unique in the limit.

Let now $L=\Eub\Gamma\setminus K_{\Eub}$ be the $\pi$-subcomplex of $\Eub\Gamma$ consisting of all those $\pi$-cells whose preimage under $q$ does not intersect $K'_\Efree$. We claim that the complement $K_{\Eub}$ is $\Gamma$-locally $\pi$-precompact. Indeed, let a non-equivariant cell of $\Eub\Gamma$ be given and let $S\subset\Gamma$ be the subset to which the cell belongs under the barycentric subdivision. The cell can only intersect $K_{\Eub}$ if there is some tuple in $S$ such that the corresponding non-equivariant cell in $\Efree\Gamma$ intersects $K'_\Efree$. By the construction of $K'_\Efree$ this can only happen if there is a tuple in $S$ of length $\leq n+1$ such that the corresponding cell intersects $K'_\Efree$. But there are only finitely many such tuples, and for each cell corresponding to one of them there are only finitely many $\pi$-cells in its $\Gamma$-orbit that intersect $K'_\Efree$. Hence, only finitely many $\pi$-cells in the $\Gamma$-orbit of the original cell intersect $K_{\Eub}$.

Finally, $q^{-1}(L)$ is a $\pi$-subcomplex of $\Efree\Gamma$. Its complement, being the preimage of a $\Gamma$-locally $\pi$-precompact set, is $\Gamma$-locally $\pi$-precompact. The restriction maps are surjective because $(\Efree\Gamma\setminus K_\Efree)\cup q^{-1}(L)\subset\Efree\Gamma\setminus K'_{\Efree}$, and we have already seen that cohomology classes can be extended to $\HZ^n_\pi(\Efree\Gamma\setminus K'_{\Efree};\Z)$ and $\HZ^{n+1}_\pi(\Efree\Gamma,\Efree\Gamma\setminus K'_{\Efree};\Z)$.
\end{proof}

\begin{lemma}\label{lem:homology-extension-to-Eub}
Let $\Efree\Gamma$ and $\Eub\Gamma$ be represented by the models above.
Let $\underline{A}\subset\underline{X}\subset\Eub\Gamma$ be any $\pi$-invariant subcomplexes \parensup{possibly $\underline{X}=\Eub\Gamma$ or $\underline{A}=\emptyset$}, let $X=q^{-1}(\underline{X})$ and $A=q^{-1}(\underline{A})$.
Assume that $n\geq 1$ and $\HZ^k(\Bfree{G};\Z)=0$ for all finite subgroups $G<\pi$ and all $0<k<n$.

Then there is a short exact sequence
\[\begin{tikzcd}0\arrow[r] &
\HZ^n_\pi(\underline{X},\underline{A};\Z)\arrow[r, "q^*"] &
\HZ^n_\pi(X,A;\Z)\arrow[r] &
Z\arrow[r] &
0\end{tikzcd},\]
where $Z$ is the group of sections from $\pi\backslash X$ into a certain sheaf $\mathcal{A}$, all of whose germs are given by $\HZ^n(\Bfree{G};\Z)$ for some finite subgroup $G<\pi$.
\end{lemma}
\begin{proof}
Because the action of $\pi$ on the coefficient module $\Z$ is trivial, $\HZ^*_\pi(X,A;\Z)$ and $\HZ^*_\pi(\underline{X},\underline{A};\Z)$ are isomorphic to $\HZ^*(\pi\backslash X,\pi\backslash A;\Z)$ and $\HZ^*(\pi\backslash \underline{X},\pi\backslash\underline{A};\Z)$, respectively. Let $\mathit{pr}^\pi_\Efree\colon\Efree\Gamma\to\pi\backslash\Efree\Gamma$ and $\mathit{pr}^\pi_{\Eub}\colon\Eub\Gamma\to\pi\backslash\Eub\Gamma$ denote the quotient maps.
We apply the Leray spectral sequence for sheaf cohomology to the induced map on the quotient $\overline{q}\colon(\pi\backslash X,\pi\backslash A)\to(\pi\backslash \underline{X},\pi\backslash\underline{A})$.

Using the terminology of Bredon \cite{Bredon} we let the \enquote{families of supports} $\Psi$ and $\Phi$ be the families of closed sets on $\pi\backslash X$ and $\pi\backslash\underline{X}$, respectively (in \cite{NitscheThesis} compact subsets were used erroneously).

By \cite{Bredon}*{Definition IV.6.1} there is a spectral sequence of sheaf cohomology groups
\[E_2^{p,r}=\HZ^p_\Phi(\pi\backslash\underline{X};\mathcal{A}^r)\Rightarrow\HZ_{\Phi(\Psi)}^{p+r}(\pi\backslash X,\pi\backslash A;\Z),\]
where $\mathcal{A}^r=\mathcal{H}_\Psi^r(\overline{q},\overline{q}\vert(\pi\backslash A);\Z))$ is the Leray sheaf, obtained from the presheaf that assigns to each open set $U\subset\underline{X}$ the abelian group
\[\HZ^r_{\Psi\cap\overline{q}^{-1}(U)}(\overline{q}^{-1}(U),\overline{q}^{-1}(U)\cap (\pi\backslash A);\Z).\]
All families of supports occurring in the above sheaf cohomology groups are simply the families of closed sets on the respective spaces. Because CW-complexes are paracompact, these families are paracompactifying in the sense of \cite{Bredon}*{Theorem I.6.1}, and it follows by \cite{Bredon}*{Theorem III.1.1} that $\HZ_{\Phi(\Psi)}^{p+r}(\pi\backslash X,\pi\backslash A;\Z)$ and $\HZ^r_{\Psi\cap\overline{q}^{-1}(U)}(\overline{q}^{-1}(U),\overline{q}^{-1}(U)\cap (\pi\backslash A);\Z)$ coincide with the usual singular cohomology groups.

We claim that the set of germs of $\mathcal{A}^r$ at a point $\mathit{pr}^\pi_{\Eub}(y)\in \pi\backslash\underline{X}$ is zero if $\mathit{pr}^\pi_{\Eub}(y)\in\pi\backslash\underline{A}$, and otherwise given by $\HZ^r(\Bfree\pi_y;\Z)$, the group cohomology of the stabilizer of $y$ in $\pi$.
To prove the claim we first find a neighborhood $\mathit{pr}^\pi_{\Eub}(y)\in U\subset\pi\backslash\underline{X}$ such that $(\mathit{pr}^\pi_{\Eub})^{-1}(U)=\bigsqcup_{\overline{g}\in\pi/\pi_y} V_{\overline{g}}$. This is possible because $\pi$ acts properly on $\underline{X}$.
Unless $\mathit{pr}^\pi_{\Eub}(y)\in\pi\backslash\underline{A}$, we also assume $U\cap\pi\backslash\underline{A}=\emptyset$.
Next, we construct, via cellular induction, a neighborhood basis $\{V'_i\}_{i\in\N}$ of $y$ such that each $V'_i$ completely lies in $(\mathit{pr}^\pi_{\Eub})^{-1}(U)$ and is star-shaped around $y$. Then $\{U'_i\}=\{\mathit{pr}^\pi_{\Eub}(V'_i)\}$ is a neighborhood basis of $\mathit{pr}^\pi_{\Eub}(y)$.
It suffices to show that $q^{-1}(y)$ and all $q^{-1}(V'_i)$ are (non-equivariantly) contractible because from this it follows that all the preimages $\overline{q}^{-1}(\mathit{pr}^\pi_{\Eub}(V'_i))=\pi_y\backslash q^{-1}(V'_i)$ are models for $\Bfree\pi_y$ and that $\varinjlim\HZ^r(\overline{q}^{-1}(U'_i);\Z)=\HZ^r(\overline{q}^{-1}(\mathit{pr}^\pi_{\Eub}(y));\Z)=\HZ^r(\Bfree\pi_y;\Z)$.

To construct the contracting homotopies, recall that the points in the infinite join model $\Efree\Gamma$ can be written in coordinate form $x=(t_j\gamma_j)_\N$ with $t_j\in[0,1]$, $\gamma_j\in\Gamma$, $\sum t_j=1$ and $t_j=0$ for all but finitely many entries. The image $q(x)$ is obtained by adding together those coefficients $t_j$ that precede the same group elements and forgetting the ordering of the group elements.

The first step is to apply the homotopy described in \cite{tomDieck}*{14.4.4}, from the identity map on $\Efree\Gamma$ to the map that sends
$(t_1\gamma_1,t_2\gamma_2,\dots)$ to $(t_1\gamma_1,0,t_2\gamma_2,0,\dots)$.
The reverse of this homotopy is obtained by stacking together an infinite number of homotopies of the form
\[
(t_1\gamma_1,\dots,t_j\gamma_j,tt_{j+1}\gamma_{j+1},(1-t)t_{j+1}\gamma_{j+1},tt_{j+2}\gamma_{j+2},(1-t)t_{j+2}\gamma_{j+2},\dots).
\]
The homotopy preserves the fibers of $q$.
The second step is to choose any preimage of $y$ under $q$ of the form $(0,t'_1\gamma'_1,0,t'_2\gamma'_2,0,\dots)$ and to apply the homotopy
$((1-t)t_1\gamma_1,tt'_1\gamma'_1,(1-t)t_2\gamma_2,tt'_2\gamma'_2,\dots)$.
Because $V'_i$ is star-shaped, this homotopy is well-defined on $q^{-1}(V'_i)$.
This finishes the proof of the claim.

Now, by assumption, $\HZ^k(\Bfree{\pi_y};\Z)=0$ for $0<k<n$.
Hence, there are only two non-vanishing entries on the $n$-th diagonal of the $E_\infty$ page of the spectral sequence, and the extension problem becomes
\[\begin{tikzcd}0\arrow[r] &
E_\infty^{n,0}\arrow[r] &
\HZ^n(\pi\backslash X,\pi\backslash A;\Z)\arrow[r] &
E_\infty^{0,n}\arrow[r] &
0\end{tikzcd}.\]
The group $E_\infty^{n,0}=E_2^{n,0}$ is equal to $\HZ^n_\Phi(\pi\backslash\underline{X};\mathcal{A}^0)$, where the sheaf $\mathcal{A}^0$ is zero over $\pi\backslash\underline{A}$ and the trivial $\Z$-sheaf everywhere else.
By \cite{Bredon}*{Proposition II.12.3}, and again \cite{Bredon}*{Theorem III.1.1}, it follows that $E_\infty^{n,0}=\HZ^n(\pi\backslash\underline{X},\pi\backslash\underline{A};\Z)$.
The group $E_\infty^{0,n}=E_2^{0,n}$ is, by definition, the group of global sections from $\pi\backslash \underline{X}$ into the sheaf $\mathcal{A}^n$ that is zero over $\pi\backslash\underline{A}$ and has germs $\mathcal{A}_{\mathit{pr}^\pi_{\Eub}(y)}=\HZ^n(\Bfree{\pi_y};\Z)$ elsewhere.

The first homomorphisms in the short exact sequence is the induced map $\overline{q}^*$, the second one is pointwise induced by the restriction to the respective fiber (see \cite{Bredon}*{IV.6.3 and Exercise 5}).
\end{proof}

In principle, we can try to use the preceding two lemmas in order to extend the equivariant transfer map for all the application scenarios considered previously. Unfortunately, the conditions on $\Gamma$ occurring in \cref{lem:homology-extension-to-Eub} are very restrictive.
In the special case of a codimension $2$ submanifold with trivial normal bundle we obtain the following.

\begin{theorem}\label{thm:application-trivialized-codim-2}
Let $E$ be a multiplicative equivariant homology theory with lf\nobreakdash-restrictions.
Suppose that we are in \cref{setup:geom}, with $N \subset M$ an embedding of codimension $k=2$.

In addition, assume that the induced map $\pi_1(N)\to\pi_1(M)$ is injective and $\pi_2(N)\to\pi_2(M)$ surjective, and that the normal bundle of $N$ is trivialized by a map $(\disk\nu,\sphere\nu)\to(\disk^2,\sphere^1)$.
Let $\theta\in E^2(\disk\nu,\sphere\nu)$ be the pullback of the twice suspended unit in $E^2(\disk^2,\sphere^1)$ under the trivialization.

Then the transfer map $\tau_\theta$ can be extended to the classifying space for free actions and further to the classifying space for proper actions. The extensions are natural with respect to transformations of homology theories \parensup{that preserve all structure}.
\end{theorem}

\begin{proof}
The twice suspended unit in $E^2(\disk^2,\sphere^1)$ is the coboundary of the once suspended unit $x\in E^1(\sphere^1)$. Hence, $\theta=\delta\circ f^*(x)$, where $f\colon\sphere\nu\to\sphere^1$ is the restriction of the trivialization map. The homotopy class of $f$ defines an element in $\HZ^1(\sphere\nu;\Z)\cong\HZ^1_\pi(\ucov{\sphere\nu})$.
As in the proof of \cref{thm:application-oriented-codim-3}, we can extend $[f]$ to $[f_\ucov{M}]\in \HZ^1_\pi(\ccomp{\ucov{M}}{\ucov{\disk\nu}^\circ})$ and to $[f_\Efree]\in \HZ^1_\pi(\Efree\Gamma\setminus K_\Efree)$, where $K_\Efree$ is $\Gamma$\nobreakdash-locally $\pi$-precompact.

Now we switch to the infinite join model for $\Efree\Gamma$ and invoke \cref{lem:adapt-K-to-model} to replace $[f_\Efree]$ with $(i_{q^{-1}(L)})^*[f'_\Efree]\in\HZ_\pi^1(q^{-1}(L);\Z)$. Note that by naturality of the equivariant transfer map, the generalized cohomology classes $\delta\circ{f_\Efree}^*(x)$, $\delta\circ(f'_\Efree)^*(x)$ and $\delta\circ((i_{q^{-1}(L)})^*f'_\Efree)^*(x)$ all give rise to the same transfer map.

Now, the conditions of \cref{lem:homology-extension-to-Eub} are satisfied for $\underline{X}=L$, $\underline{A}=\emptyset$ and $n=1$. Because $\HZ^1(\Bfree{G};\Z)=0$ for all finite groups $G$, it follows that the induced map $q^*\colon\HZ^1_\pi(L;\Z)\to\HZ^1_\pi(q^{-1}(L);\Z)$ is an isomorphism. Hence, we can extend $(i_{q^{-1}(L)})^*[f'_\Efree]$ to $[f_{\Eub}]\in \HZ^1_\pi(L;\Z)$.

Finally, the extensions of the transfer map are obtained by applying \cref{thm:transfer-extension}, where the extensions of $\theta$ are given by $\delta\circ{f_\Efree}^*(x)$ and $\delta\circ{f_{\Eub}}^*(x)$.
\end{proof}

As discussed in the introduction, the most interesting application of \cref{thm:application-trivialized-codim-2} is the case where $E=\KO$ is real K-homology, where our extension provides some context for \cref{theo:obstruc_codim_2}.
We expect that the extension is compatible with the transfer map defined by
Kubota in\cite{Kubota} on the K-theory of the maximal group $\Cstar$-algebras,
but we do not know whether this is true.

\end{document}

{\small
\bibliographystyle{plain}
\bibliography{whatever}
}
